\documentclass[12pt,oneside]{ip-journal}
\usepackage{geometry}
\usepackage{extarrows}
\usepackage{epstopdf}
\usepackage{mathpazo}
\usepackage{mathtools}
\usepackage[cmtip,all]{xy}
\usepackage{mathrsfs}
\usepackage{amscd}
\usepackage{amsmath}
\usepackage{amssymb}
\usepackage{amsthm}
\usepackage{float}
\usepackage[dvips]{graphicx}
\usepackage{color}
\usepackage{tikz} 
\usepackage{enumerate}
\usepackage{appendix}
\usetikzlibrary{matrix,arrows} 
\usepackage{array}
\usepackage{amscd}
\usepackage[all]{xy}
\usepackage{enumitem}

\usepackage{array}
\usepackage{amscd}
\usepackage[all]{xy}

\usepackage{array}
\usepackage{amscd}
\usepackage[all]{xy}

\DeclareFontFamily{U}{rsfs}{%
\skewchar\font127}
\DeclareFontShape{U}{rsfs}{m}{n}{%
<-6>rsfs5<6-8.5>rsfs7<8.5->rsfs10}{}
\DeclareSymbolFont{rsfs}{U}{rsfs}{m}{n}
\DeclareSymbolFontAlphabet
{\mathrsfs}{rsfs}
\DeclareRobustCommand*\rsfs{%
\@fontswitch\relax\mathrsfs}

\DeclareFontFamily{U}{rsfs}{%
\skewchar\font127}
\DeclareFontShape{U}{rsfs}{m}{n}{%
<-6>rsfs5<6-8.5>rsfs7<8.5->rsfs10}{}
\DeclareSymbolFont{rsfs}{U}{rsfs}{m}{n}
\DeclareSymbolFontAlphabet             
{\mathrsfs}{rsfs}
\DeclareRobustCommand*\rsfs{%
\@fontswitch\relax\mathrsfs}

\theoremstyle{plain}

\newtheorem{thm}{Theorem}[section]
\newtheorem{prop}[thm]{Proposition}
\newtheorem{lem}[thm]{Lemma}

\newtheorem{rmk}[thm]{Remark}
\newtheorem{cor}[thm]{Corollary}

\newtheorem*{prop*}{Proposition}
\newtheorem*{notn}{Notation}

\newtheorem*{thr}{Theorem}
\newtheorem{prop-defi}[thm]{Proposition-Definition}
\newtheorem{thm-defi}[thm]{Theorem-Definition}
\newtheorem{lem-defi}[thm]{Lemma-Definition}

\newtheorem{conj}[thm]{Conjecture}

\newcommand{\C}{\mathbb{C}}

\newcommand{\Z}{\mathcal{Z}}
\newcommand{\I}{\mathcal{I}}

\newcommand{\CO}{{\mathcal{O}}}

\renewcommand{\L}{{\boldsymbol{L}}}
\renewcommand{\S}[1]{S^{[#1]}}
\newcommand{\z}[1]{\Z^{[#1]}}
\renewcommand{\i}[1]{\I^{[#1]}}

\newcommand{\dR}{\mathbf{R}}

\renewcommand{\hom}{\mathscr{H}om}

\newcommand{\Coh}{\operatorname{Coh}}

\newcommand{\eE}{\mathbb{E}}

\newcommand{\ch}{\operatorname{ch}}
\newcommand{\DT}{\operatorname{DT}}

\newcommand{\vir}{\operatorname{vir}}

\renewcommand{\v}[1]{V^{[#1]}}

\newcommand{\sP}{\mathsf{P}}
\newcommand{\sQ}{\mathsf{Q}}
\newcommand{\sA}{\mathsf{A}}

\newcommand{\IP}{\mathbb{P}}

\renewcommand{\t}{\mathbf{t}}
\newcommand{\s}{\mathbf{s}}

\newcommand{\sK}{\mathsf{K}}
\newcommand{\sT}{\mathsf{T}}
\newcommand{\sE}{\mathsf{E}}

\makeatletter
\newcommand{\subalign}[1]{%
  \vcenter{%
    \Let@ \restore@math@cr \default@tag
    \baselineskip\fontdimen10 \scriptfont\tw@
    \advance\baselineskip\fontdimen12 \scriptfont\tw@
    \lineskip\thr@@\fontdimen8 \scriptfont\thr@@
    \lineskiplimit\lineskip
    \ialign{\hfil$\m@th\scriptstyle##$&$\m@th\scriptstyle{}##$\crcr
      #1\crcr
    }%
  }
}
\makeatother

\newcommand{\be}{\begin{equation}}
\newcommand{\ee}{\end{equation}}
\newcommand{\IG}{{\mathbb{G}}}

\newcommand\IZ{\mathbb {Z}}

\newcommand\IQ{\mathbb {Q}}
\newcommand{\IC}{\mathbb{C}}
\newcommand{\IE}{\mathbb{E}}
\newcommand{\IR}{\mathbb{R}}

\newcommand{\CU}{{\mathcal U}}
\newcommand{\ba}{\begin{array}}
\newcommand{\ea}{\end{array}}

\newcommand{\CX}{{\mathcal X}}
\newcommand{\IF}{{\mathbb F}}

\newcommand{\CK}{{\mathcal K}}

\newcommand{\bal}{\begin{aligned}}
\newcommand{\eal}{\end{aligned}}

\newcommand{\oF}{{\overline {\mathscr{F}}}}

\newcommand{\longto}{\longrightarrow}

\newcommand{\CE}{{\mathscr E}}
\newcommand{\CA}{{\mathcal A}}

\newcommand{\CF}{{\mathscr F}}

\newcommand{\CM}{{\mathcal M}}
\newcommand{\CL}{{\mathcal L}}

\newcommand{\IL}{{\mathbb L}}

\newcommand{\CC}{{\mathcal C}}
\newcommand{\CI}{{\mathcal I}}

\newcommand{\CQ}{{\mathcal Q}}

\newcommand{\CT}{{\mathcal T}}

\newcommand{\IA}{{\mathbb A}}

\newcommand{\CY}{{\mathcal Y}}

\newcommand{\Higgs}{\operatorname{Higgs}}
\newcommand{\xlongto}{\xlongrightarrow} 

\title[Atiyah class and sheaf counting on local fourfolds]{Atiyah class and sheaf counting on local Calabi Yau fourfolds}

\author{Duiliu-Emanuel Diaconescu and Artan Sheshmani and Shing-Tung Yau}

\begin{document}

\address{Duiliu-Emanuel Diaconescu, Rutgers University, New High Energy Theory Center, Department of Physics, Rutgers, The State University Of New Jersey
126 Frelinghuysen Rd., Piscataway, NJ 08854-8019, USA}

\address{Artan Sheshmani, Center for Mathematical Sciences and Applications, Harvard University, Department of Mathematics, 20 Garden Street, Room 207, Cambridge, MA, 02138, USA}
\address{Centre for Quantum Geometry of Moduli Spaces, Aarhus University, Department of Mathematics, Ny Munkegade 118, building 1530, 319, 8000 Aarhus C, Denmark}
\address{National Research University Higher School of Economics, Russian Federation, Laboratory of Mirror Symmetry, NRU HSE, 6 Usacheva str., Moscow, Russia, 119048}
\email{artan@cmsa.fas.harvard.edu}

\address{Shing-Tung Yau, Harvard University, Department of Mathematics, Cambridge, MA, 02138, USA}
\email{yau@math.harvard.edu}

\maketitle

\begin{abstract}
We discuss Donaldson-Thomas (DT) invariants of torsion sheaves with 2 dimensional support on a smooth projective surface in an ambient non-compact Calabi Yau fourfold given by the total space of a rank 2 bundle on the surface. We prove that in certain cases, when the rank 2 bundle is chosen appropriately, the universal truncated Atiyah class of these codimension 2 sheaves reduces to one, defined over the moduli space of such sheaves realized as torsion codimension 1 sheaves in a noncompact divisor (threefold) embedded in the ambient fourfold. Such reduction property of universal Atiyah class enables us to relate our fourfold DT theory to a reduced DT theory of a threefold and subsequently then to the moduli spaces of sheaves on the base surface using results in \cite{GSY17b,GSY17a}. We finally make predictions about modularity of such fourfold invariants when the base surface is an elliptic K3.
\end{abstract}

\setcounter{tocdepth}{3}
\tableofcontents

\section{Introduction} 

This paper is motivated by two recent developments in Donaldson-Thomas theory, namely the mathematical theory of 
Vafa-Witten invariants of complex surfaces constructed in \cite{TT17a,TT17b}, \cite{GSY17a, GSY17b} and \cite{GK17}, and the construction  \cite{BJ17,CL14,CL17} 
of virtual cycle invariants for sheaves on Calabi-Yau fourfolds. In the approach of \cite{GSY17b, TT17a} the Vafa-Witten invariants of a smooth projective surface $S$ are constructed as reduced virtual cycle invariants of two dimensional sheaves on the total space of the canonical bundle $K_S$. In fact, the construction of loc. cit. is more general, including  equivariant residual invariants of 
two dimensional sheaves on the total space $Y$ of an arbitrary 
line bundle $L$ on $S$. The starting point for the present study is the observation that the total space $X$ of the rank two bundle $L\oplus  K_S\otimes L^{-1}$ is a Calabi-Yau fourfold, which is at the same time isomorphic to the total space of the canonical line bundle $K_Y$. This leads to the natural question whether the invariants constructed in 
\cite{GSY17b,TT17a} for $Y$ are related to fourfold Donaldson-Thomas invariants. In fact such a correspondence has already been proven in \cite{CL14} for the case where $Y$ is a compact Fano threefold, 
using a local Kuranishi map construction for fourfold Donaldson-Thomas invariants.
The main goal of this paper is prove an analogous correspondence for more general situations allowing $Y$ to be a quasi-projective threefold with a suitable torus action. As discussed in detail below, the approach employed here is different, using the truncated Atiyah class formalism of \cite{HT10}. A second goal 
is to test the modularity conjecture for two dimensional sheaf counting invariants in the local fourfold framework. This will be 
carried out for a particular class of examples, where $Y$ is the total space of certain line bundles $L$ over an elliptic K3 surface $S$. 
 
The main motivation for this study resides in modular properties of Donaldson-Thomas invariants for sheaves
with two dimensional support. This has been a central theme in enumerative geometry starting with the 
early work of \cite{Betti_Hilbert, KY_betti_I, KY_betti_II, KY_S_duality, KY_euler, GL_theta} on torsion free sheaves on surfaces. It has been also studied intensively for two dimensional sheaves on Calabi-Yau threefolds \cite{DT_twodim, Gen_DT_twodim, Toda_flops, GST13, Toda_An, Toda_Ptwo}. Further results on this topic include 
\cite{GK17,GS17b,GS17a,TT17a}. As modularity is a natural consequence of S-duality \cite{VW94,DM11,Farey_tail,M5_genus}, 
similar physical arguments predict that it should be also manifest in the Donaldson-Thomas theory 
of two dimensional sheaves on  Calabi-Yau fourfolds. At the moment this question is completely open. This paper essentially consists of two parts, as explained below. \\

\subsection{General results}\label{foundations}  
Let $Y$ be a smooth quasi-projective variety of dimension $d-1$, $d\geq 2$, and let $X$ be the total space of the canonical bundle $K_Y$. Let 
$\CM$ be a quasi-projective fine moduli scheme of compactly supported stable sheaves $\mathscr{F}$ on $X$. 
Suppose that all sheaves, $\mathscr{F}$, parameterized by $\CM$ are scheme theoretically supported on $Y$. The formalism of truncated Atiyah classes developed in \cite{HT10}, yields two natural (unreduced) ``\textit{deformation-obstruction theories}",  $a_X:{\mathbb{E}^{\bullet}}_X\to \IL^{\bullet}_\CM$ and 
$a_Y: {\mathbb{E}^{\bullet}}_Y\to \IL^{\bullet}_\CM$ for 
 the moduli space $\CM$. These are associated respectively to deformations of sheaves, realized, respectively, as modules over the structure rings of $X$ and $Y$. As in \cite{BF97},
in the present paper an obstruction theory will be an object ${\mathbb{E}^{\bullet}}$ of amplitude $[-1,\ 0]$ in the derived category of the moduli space, together with a morphism in derived category 
$\phi : {\mathbb{E}^{\bullet}}\to \IL^{\bullet}_\CM$ such that $h^0(\phi)$ is an isomorphism and $h^{-1}(\phi)$ is an epimorphism of sheaves. Here $\IL^{\bullet}_\CM$ denotes the $[-1, 0]$ truncation of the full cotangent complex of $\CM$ as in \cite{HT10}. 
The first general result proven in this paper is 
\begin{thm}\label{obsthm}
(Proposition \ref{Aclasslocal} and 
 Proposition \ref{obsprop}). There is a natural splitting in $\mathscr{D}^b(\CM)$: ${\mathbb{E}^{\bullet}}_X \cong {\mathbb{E}^{\bullet}}_Y \oplus {\mathbb{E}^{\bullet}}_Y^\vee[1-d]$. Moreover, the obstruction theory $a_{X}$ factors through ${\mathbb{E}^{\bullet}}_Y$ and there is a direct sum decomposition $$a_X = (a_Y, 0).$$ 
\end{thm}

This shows that, under the above assumptions, the fourfold obstruction theory  of $\CM$ admits a reduction to the threefold obstruction theory. Although this is a natural statement, the proof of it, using the truncated Atiyah classes, is surprisingly long. In particular, it uses an explicit calculation, Theorem \ref{Apropone}, for the relative Atiyah class of a product.

 Using this result, the proof of Theorem \ref{obsthm} is then given subsequently in Sections \ref{compAclass} and \ref{compobs}, the main 
 intermediate results being Propositions \ref{Aclasslocal} and 
 \ref{obsprop} respectively. In particular, Propositions \ref{obsprop} is a more 
 explicit restatement of Theorem \ref{obsthm}. 
 
Applying Theorem \ref{obsthm}, a reduced equivariant Donaldson-Thomas theory is constructed in Section \ref{DTsect} for two dimensional sheaves on certain local fourfolds. In section \ref{DTsect}, $S$ is a smooth projective surface 
with nef canonical bundle, $K_S$, equipped with a polarization $h = c_1(\CO_S(1))$. It is also assumed that $H^1(\CO_S)=0$ and 
the integral cohomology of $S$ is torsion free. 
Given an effective divisor $D$ on $S$, a fourfold $X$ is constructed as the 
total space of a rank two bundle $K_S(D)\oplus \CO_S(-D)$. The threefold $Y$ is the total space of the line bundle $L=K_S(D)$ on $S$, which is canonically embedded in $X$ as a divisor. 

As in \cite{GSY17b} the topological invariants of a compactly supported two dimensional sheaf $\mathscr{F}$ on $X$ are completely determined by the total Chern class $c(g_*\mathscr{F})$ of the direct image $g_*F$, where $g: X \to S$ is the canonical projection. 
In turn, the latter is naturally given by a triple 
\be\label{eq:gammaF}
\gamma(\mathscr{F}) = (r(\mathscr{F}), \beta(\mathscr{F}), n(\mathscr{F})) \in \IZ \oplus H_2(S, \IZ)\oplus \IZ
\ee
The entry, $r(\mathscr{F})={\rm rk}(g_*\mathscr{F})$, will be called the rank  
of $\mathscr{F}$ in the following. 
Moreover, since $D$ is effective nonzero divisor, one easily shows (Corollary \ref{supplemmaA}) 
that any such $h$-stable sheaf $\mathscr{F}$ is scheme theoretically supported on $Y \subset X$. 
This yields a natural identification of the moduli spaces $\CM_h(X,\gamma)\cong \CM_h(Y,\gamma)$. Moreover, as proven in 
\cite[Proposition 2.4]{GSY17b} and \cite[Theorem 6.5]{TT17a} 
the latter has a {\it reduced} perfect obstruction theory
$a_Y^{\rm red}: {\mathbb{E}^{\bullet}}_Y^{\rm red} \to \IL^{\bullet}_\CM$.
The detail of this construction is reviewed in Section 
\ref{redobs} for completeness. 
Then Theorem \ref{obsthm} implies: 
\begin{cor}\label{obscor} 
The fourfold obstruction theory $a_X:{\mathbb{E}^{\bullet}}_X \to \IL^{\bullet}_\CM$ admits a reduction to the reduced perfect obstruction theory 
$a_Y^{\rm red}: {\mathbb{E}^{\bullet}}_Y^{\rm red} \to \IL^{\bullet}_\CM$ constructed 
in \cite[Proposition 2.4]{GSY17b} and \cite[Theorem 6.5]{TT17a}. 
\end{cor} 

Using this result, the reduced equivariant Donaldson-Thomas invariants for two dimensional sheaves on $X$ are defined by residual virtual integration in Section \ref{redDT}. The next two sections, \ref{rktwosect} and \ref{kthreesect}, are focused on explicit 
computations and modularity conjectures for rank two invariants. 
As shown in 
\cite[Section 3]{GSY17b} and \cite[Section 7]{TT17a}, there are two types of rank two torus fixed loci. The ``\textit{type I}" fixed locus is 
isomorphic to the moduli space $\CM_h(S,\gamma)$ of torsion free 
sheaves on $S$ with topological invariants $\gamma$. The ``\textit{type II}" 
fixed loci consist of sheaves $\mathscr{F}$ with scheme theoretic support 
on the divisor $2S \subset Y$, whose connected components are naturally isomorphic to 
twisted nested Hilbert schemes, as reviewed in detail in Section 
\ref{fixedloci}. 
By analogy with similar previous results \cite{GSY17b, GK17}, universality results for residual virtual integrals of type I and II are proven in Propositions \ref{univI} and \ref{univII}. These are valid for any pair $(S,D)$ as above. Explicit results and a modularity conjecture are then obtained in Section \ref{kthreesect} for elliptic K3 surfaces 
with $D$ a positive multiple of the fiber class, $D=mf$. \\

\subsection{Rank two results and conjectures for elliptic K3 surfaces} 

Let $S$ be a smooth generic elliptic K3 surface in Weierstrass form. By the genericity assumption, all elliptic fibers are reduced and irreducible, and have at most nodal singularities. Furthermore, the Mordell-Weill group is trivial and the Picard group of $S$ is freely generated by the fiber class $\mathscr{F}$ and the canonical section class $\sigma$. The K\"ahler cone of $S$ then consists of real divisor classes $h = t\sigma + uf$ with $t, u\in \IR$, 
$0< 2t <u$. 

Let $D=mf$, $m\geq 1$, be a positive multiple of the elliptic fiber so that the local fourfold $X$ is the total space of $\CO_S(mf) \oplus \CO_S(-mf)$. 
Moreover, let the topological invariant \eqref{eq:gammaF} be of the form 
$\gamma=(2,f, n)$ with $n \in \IZ$. The Bogomolov inequality, proven in Lemma \ref{bogomolov}, shows that in this case the moduli space $\CM_h(Y,\gamma)$ is empty for $n<0$. Moreover, the global boundedness result proven in Lemma \ref{nowalls} shows that there are no marginal stability walls associated to sheaves $\mathscr{F}$ with invariants $\gamma$ for 
\be\label{eq:smallfchamber}
0<{t\over u}< {1\over 1+8n}.
\ee
Therefore for any fixed $n\geq 0$ the equivariant residual Donaldson-Thomas invariants $DT_h(X,\gamma)$ reach a stable limit, $D_\infty(X,\gamma)\in \IQ({\bf s})$ as the size of the elliptic fibers decreases, keeping the size of the section fixed. Then 
let
\[ 
Z_\infty(X, 2,f;q) = \sum_{n \geq 0} DT_\infty(X,2,f,n) q^{n-2} \in q^{-2}\IQ({\bf s})[[q]]
\]
be the generating function of  limit invariants. 
Clearly the above partition functions splits into type I and type II contributions denoted by $Z_\infty(X, 2,f;q)_I$ and $Z_\infty(X, 2,f;q)_{II}$ respectively. Then the following result is proven in Section \ref{partfctsect}.
\begin{prop}\label{DTkthreeI} 
The series $Z_\infty(X, 2,f;q)_I$ belongs to $\IQ({\bf s})[[q]]\subset \IQ({\bf s})[[q^{-1}, q]]$ and the following identity holds in the quadratic extension $\IQ({\bf s})[[q^{1/2}]]$: 
\be\label{eq:typeIformula} 
Z_\infty(X, 2,f;q)_I = {{\bf s}^{-1}\over 2}(\Delta^{-1}(q^{1/2}) + \Delta^{-1}(-q^{1/2})),
\ee
where $\Delta(q)$ is the discriminant modular form. 
\end{prop}
Moreover, 
Proposition 
\ref{DTkthreeII} shows that all type II nonzero contributions 
vanish for $m$ even. At the same time for $m$ odd, all nonzero type II contributions are associated to nested Hilbert schemes without divisorial twists. Then the following conjecture is inferred from the results of 
\cite[Section 8]{TT17a} using the universality result Proposition \ref{univII}. 
\begin{conj}\label{typeIIconj} 
Suppose $m\geq 1$ is odd. Then 
\be\label{eq:typeIIgenfct} 
Z_\infty(X, 2,f;q)_{II} = {{\bf s}^{-1}\over 4}\Delta^{-1}(q^2).
\ee
\end{conj}
\textbf{Plan of the paper}. The key theorem of the current article is the reduction Theorem  \ref{obsthm}. We approach the problem by essentially proving a statement that the universal (relative) Atiyah class of the (universal) sheaf $\IF$ on the moduli space, $\CM_h(X,\gamma)$ of stable torsion sheaves $\mathscr{F}$ on $X$, with scheme theoretic support on $Y$, has a reduction to the universal Atiyah class of the universal sheaf $\mathbb{G}$ on $\CM_h(Y,\gamma)$ (which is isomorphic to  $\CM_h(X,\gamma)$) parameterizing stable sheaves $\mathscr{G}$ on $Y$ (where for all $\mathscr{F}\in \CM_h(X,\gamma)$ we have $\mathscr{F}\cong i_{*}\mathscr{G}, i:Y\hookrightarrow X$). This will be done in Proposition \ref{Aclasslocal}. Our proof of Proposition \ref{Aclasslocal} relies heavily on a certain splitting property of the universal relative Atiyah classes on product schemes, as stated in Theorem \ref{Apropone}. Hence, we start in Section \ref{Aclassect} a detailed discussion of Atiyah classes on product schemes which can be read independently from the rest of the paper. Then, in Section \ref{obstrsect} we use Theorem \ref{Apropone} to prove Proposition \ref{Aclasslocal} and later on, as an application of Proposition \ref{Aclasslocal}, we show in Proposition \ref{obsprop} that the reduction of Atiyah classes as above, leads to a reduction in the level of deformation obstruction theories, as stated in Theorem \ref{obsthm}. The latter proves that in our current geometric setup the Donaldson-Thomas invariants of fourfold $X$ can be completely recovered from those of threefold $Y$ which, after being reduced, enjoy modularity property, discussions of which we have included in Sections \ref{DTsect}, \ref{rktwosect} and \ref{kthreesect}.

\section*{Aknowledgement}
D.-E. D would like to thank Ron Donagi and Tony Pantev for their interest in this work and very helpful discussions. 
D.-E.D. was partially supported by NSF grants 
DMS-1501612 and DMS-1802410. 
  A.S. would like to thank Amin Gholampour, Alexey Bondal, Alexander Efimov for helpful discussions and commenting on first versions of this article. A.S. was partially supported by NSF DMS-1607871, NSF DMS-1306313 and Laboratory of Mirror Symmetry NRU HSE, RF Government grant, ag. No 14.641.31.0001. A.S. would like to further sincerely thank the Center for Mathematical Sciences and Applications at Harvard University, the center for Quantum Geometry of Moduli Spaces at Aarhus University, and the Laboratory of Mirror Symmetry in Higher School of Economics, Russian federation, for the great help and support. S.-T. Y. was partially supported by NSF DMS-0804454,
NSF PHY-1306313, and Simons 38558. 

\section{Universal relative Atiyah classes on product schemes}\label{Aclassect}

As explained in the plan of the paper, this section is a 
self-contained account of relative truncated Atiyah classes, the main result being Theorem \ref{Apropone} below. 

\subsection{Truncated Atiyah class of a scheme}
Let $W$ be an arbitrary scheme of finite type over $\IC$ and $W\subset A$ a closed embedding in a smooth scheme $A$. 
As in \cite[Definition. 2.1]{HT10}, the truncated cotangent complex 
$\IL_W$ is the two term complex of amplitude $[-1,\ 0]$$$\mathscr{J}_W/\mathscr{J}_W^2 \to \Omega_A^1|_W,$$ where $\mathscr{J}_W\subset \CO_A$ is the defining ideal sheaf 
of $W\subset A$. Let us denote by $\Delta_W: W \to W \times W$ the diagonal morphism, by $\CO_{\Delta_W}$ 
 the extension by zero of the structure sheaf of the diagonal, and by  $\mathcal{I}_{\Delta_W}\subset \CO_{W\times W}$ the ideal sheaf of the diagonal. 
Then the following results are proven in \cite[Lemma 2.2, Definition 2.3]{HT10}.

\begin{lem}\label{AclasslemA} 
 \begin{enumerate}[label=(\roman*)]
\item There is an exact sequence of $\CO_{W\times W}$-modules 
\be\label{eq:diagseqA} 
0 \to \Delta_{W*}(\mathscr{J}_W/\mathscr{J}_W^2) {\buildrel {\tau_W}\over \longto}  \mathcal{I}_{\Delta_A}|_{W\times W} {\buildrel \xi_W \over \longto} 
\mathcal{I}_{\Delta_W} \to 0 
\ee
where the morphism $\xi_W$ is the natural projection. 

\item Let $\iota_W:\mathcal{I}_{\Delta_W}\to  \CO_{W\times W}$ denote the canonical inclusion. Then
there is a commutative diagram of $\CO_{W\times W}$-modules 
\be\label{eq:AclassdiagA}
\xymatrix{ 
\Delta_{W*}(\mathscr{J}_W/\mathscr{J}_W^2) \ar[rr]^-{{\tau_W}} \ar[d]^-{{\bf 1}} & &\mathcal{I}_{\Delta_A}|_{W\times W} 
\ar[r]^{\iota_W \circ \xi_W} \ar[d]^-{r_W}& \CO_{W\times W} \\ 
\Delta_{{W}*}(\mathscr{J}_W/\mathscr{J}_W^2) \ar[rr]^-{\Delta_{W*}(d_{W/A})} & & (\mathcal{I}_{\Delta_A}/\mathcal{I}_{\Delta_A}^2)|_{W\times W}. & \\}
\ee
where the right vertical map $r_W$ is the natural projection and the bottom row 
is naturally isomorphic to $\IL_{W}[1]$. 
\end{enumerate}\end{lem} 
By construction, the top row of diagram \eqref{eq:AclassdiagA}
is quasi-isomorphic to $\CO_{\Delta_W}$. Hence, the Diagram \eqref{eq:AclassdiagA} determines 
a morphism 
$\alpha_{W}: \CO_{\Delta_W}\to \Delta_{W*}\IL_W[1]$. The latter is the truncated universal Atiyah class of $W$. 
According to  
\cite[Definition2.1]{HT10}, the truncated cotangent complex and truncated Atiyah class are independent of choice of $A$. As shown in \cite[Definition 2.3]{HT10}, an alternative presentation of the truncated Atiyah class is obtained from the 
proof of Lemma 2.2 in loc. cit., as follows. 

\begin{lem}\label{AclasslemB} 
 \begin{enumerate}[label=(\roman*)]

\item There is a canonical isomorphism 
\be\label{eq:WAisomA} 
\mathcal{I}_{\Delta_A}|_{W\times A} \cong \mathcal{I}_{\Delta_W \subset W\times A} 
\ee
where the right hand side denotes the defining ideal sheaf for the 
closed embedding $\Delta_W \subset W\times A$. 

\item The isomorphism \eqref{eq:WAisomA} determines a further canonical isomorphism of $\CO_{W\times A}$-modules 
\be\label{eq:WAisomB}
(\CO_W\boxtimes \mathscr{J}_W)\otimes_{W\times A} \CO_{\Delta_W} {\buildrel \sim\over \longto}  \Delta_{W*}(\mathscr{J}_W/\mathscr{J}_W^2)
\ee
and a natural injection 
\be\label{eq:WAinjA} 
\CO_W\boxtimes \mathscr{J}_W \hookrightarrow \mathcal{I}_{\Delta_{A}}|_{W\times A} 
\ee
which fit in a commutative diagram of morphisms of $\CO_{W\times A}$-modules:
\be\label{eq:AclassdiagB}
\xymatrix{
\CO_W\boxtimes \mathscr{J}_W \ar@{^{(}->}[rr] \ar[d] 
& & \mathcal{I}_{\Delta_A}|_{W\times A}\ar[d] \\
(\CO_W\boxtimes \mathscr{J}_W)\otimes_{W\times A} \CO_{\Delta_W}
\ar[d]^-{\wr}
\ar[rr]^-{\nu_W} & & \mathcal{I}_{\Delta_A}|_{W\times W}\ar[d]^-{r_W} \\
\Delta_{W*}(\mathscr{J}_W/\mathscr{J}_W^2) \ar[rr]^-{\Delta_{W*}(d_{W/A})} \ar[urr]^-{{\tau_W}} & & (\mathcal{I}_{\Delta_A}/\mathcal{I}^2_{\Delta_A})|_{W\times W} \\}
\ee
Moreover, the two upper vertical arrows in the above diagram are the natural projections. 
\end{enumerate}
\end{lem}
\begin{rmk} By a slight abuse of notation, in the above lemma $\Delta_W$ indicates the canonical closed embedding $W \hookrightarrow W \times A$ as opposed to the diagonal embedding $W \hookrightarrow W \times W$. This notation will be consistently used throughout this section, the distinction being clear from the context. 
\end{rmk}

Next note the following immediate corollary of Lemma \ref{AclasslemB}. 

\begin{cor}\label{AclasscorA} 
The universal Atiyah class is equivalently defined by the commutative diagram 
\be\label{eq:AclassdiagC}
\xymatrix{ 
(\CO_W\boxtimes \mathscr{J}_W)\otimes_{W\times A} \CO_{\Delta_W}
 \ar[rr]^-{\nu_W} \ar[d]^-{{\bf 1}} & &\mathcal{I}_{\Delta_A}|_{W\times W} 
\ar[r]^{\iota_W \circ \xi_W} \ar[d]^-{r_W}& \CO_{W\times W} \\ 
(\CO_W\boxtimes \mathscr{J}_W)\otimes_{W\times A} \CO_{\Delta_W}
\ar[rr]^-{r_W\circ \nu_W} & & (\mathcal{I}_{\Delta_A}/\mathcal{I}_{\Delta_A}^2)|_{W\times W}. & \\}
\ee
where the left column is restricted to $W\times W$. As in Lemma \ref{AclasslemA}, the right vertical map is the natural projection. 
\end{cor}

\subsection{Atiyah classes on product schemes}\label{Aclassproduct} 

In most applications to moduli problems $W$ will be a product scheme $Z\times T$, where $Z$ is a smooth quasi-projective variety and $T$ is a test scheme, or the moduli space. In the latter case, as shown in 
\cite[Thm. 4.1]{HT10}, the construction of the obstruction theory of the moduli space, 
uses only the relative component $\alpha_{W/Z}$ of the Atiyah class (as will be recalled below). The goal of the section is to prove Theorem \ref{Apropone} below, which relates $\alpha_{W/Z}$ to the Atiyah 
class $\alpha_T$ of $T$. \\

To fix the notation, for any product $W_1 \times W_2$ of schemes over $\IC$ let $\pi_{W_i}$, $1\leq i \leq 2$, denote the canonical projections. Moreover, for any test scheme $T$ 
let $Z_T = Z\times T$ and, given a morphism of schemes $f: Z_1 \to Z_2$, let $f_T: (Z_1)_T \to (Z_2)_T$ denote the morphism $f \times {1}_T$.
The subscripts will be suppressed if $T={\rm Spec}(\IC)$. 
Given two complexes of sheaves 
$C^{\bullet}_1, C^{\bullet}_2$ on $W_1, W_2$, the  tensor product $\pi_{W_1}^*C^{\bullet}_1 \otimes \pi_{W_2}^*C^{\bullet}_2$ will be denoted by $C^{\bullet}_1 \boxtimes C^{\bullet}_2$. Similarly, the direct sum $\pi_{W_1}^*C^{\bullet}_1 \oplus \pi_{W_2}^*C^{\bullet}_2$ will be denoted by $C^{\bullet}_1 \boxplus C^{\bullet}_2$. Some background results 
needed below 
have been recorded in the Appendix \ref{background}. 

Let $W=Z_T$ 
where $Z$ is a smooth complex quasi-projective variety and $T$ is an arbitrary quasi-projective scheme over $\IC$. Then there is a canonical splitting 
\be\label{eq:Lsplitting}
\IL_{Z_T} \cong \Omega_Z\boxplus \IL_T.
\ee
In the present context this splitting follows from equation \eqref{eq:splitseq} proven in Corollary \ref{splitcor},
which shows that for any quasi-projective schemes $W_1,W_2$ 
there is a natural 
isomorphism 
\[ 
{{\bar \rho_W}}: \mathcal{I}_{\Delta_{W_1}}/\mathcal{I}_{\Delta_{W_1}}^2 \boxtimes 
\CO_{\Delta_{W_2}} \oplus \CO_{\Delta_{W_1}} \boxtimes 
\mathcal{I}_{\Delta_{W_2}}/\mathcal{I}_{\Delta_{W_2}}^2 {\buildrel \sim \over \longto} 
\mathcal{I}_{\Delta_{W}}/\mathcal{I}_{\Delta_{W}}^2. 
\]
Since $T$ is quasi-projective, there is a closed embedding 
$T \subset U$ into a smooth quasi-projective scheme $U$. 
Hence one can pick $A= Z_U$ 
in the construction of the Atiyah class, since $Z$ is smooth. 
Then, setting $W_1=Z$ and $W_2=U$ in the above equation yields the 
splitting \eqref{eq:Lsplitting}.  

The relative Atiyah class is the component  
$\alpha_{Z_T/Z}:  \CO_{\Delta_{Z_T}}\to \Delta_{Z_T*}(\pi^*_T\IL_T[1])$ of $\alpha_{Z_T}$ with respect to this splitting. 
Using Corollary \ref{flatcorA}, this is canonically identified with a morphism 
$$\alpha_{Z_T/Z}:  \CO_{\Delta_Z} \boxtimes \CO_{\Delta_T} \to \CO_{\Delta_Z}\boxtimes \IL_T[1].$$ The notation can be preserved with no risk of confusion.
Note also that the Atiyah class of $T$ is a morphism $\alpha_T : \CO_{\Delta_T} \to 
\Delta_{T*}\IL_T[1]$. 
Then the main result of this section is to relate the latter two classes as follows:
\begin{thm}\label{Apropone}
Let $Z,T$ be quasi-projective schemes over $\IC$ with $Z$ smooth. 
Then there is a relation 
\be\label{eq:Aclassformula} 
\alpha_{Z_T/Z} = {\bf 1}_{\CO_{\Delta_Z}} \boxtimes \alpha_T.
\ee
\end{thm}

The proof will use diagram \eqref{eq:AclassdiagC} 
in Corollary \ref{AclasscorA} and will proceed in several steps.  
Since $T$ is quasi-projective, 
there is a closed embedding $T\subset U$ with $U$ smooth quasi-projective. Moreover, as already noted below equation \eqref{eq:Lsplitting} one can set $A= Z_U$ 
in the construction of the Atiyah class since $Z$ is smooth. Then diagram \eqref{eq:AclassdiagC} reads 
\be\label{eq:AclassdiagCB}
\xymatrix{ 
(\CO_{Z_T}\boxtimes \mathscr{J}_{Z_T})\otimes_{{Z_T}\times {Z_U}} \CO_{\Delta_{Z_T}}
 \ar[rr]^-{\nu_{Z_T}} \ar[d]^-{{\bf 1}} & &\mathcal{I}_{\Delta_{Z_U}}|_{{Z_T}\times {Z_T}} 
\ar[r]^{\iota_{Z_T} \circ \xi_{Z_T}} \ar[d]^-{r_{Z_T}}& \CO_{{Z_T}\times {Z_T}} \\ 
(\CO_{Z_T}\boxtimes \mathscr{J}_{Z_T})\otimes_{{Z_T}\times {Z_U}} \CO_{\Delta_{Z_T}}
\ar[rr]^-{r_{Z_T}\circ \nu_{Z_T}} & & (\mathcal{I}_{\Delta_{Z_U}}/\mathcal{I}_{\Delta_{Z_U}}^2)|_{{Z_T}\times {Z_T}}\, ,& \\}
\ee
where the left hand column is restricted to $Z_T \times Z_T$. 
The strategy to prove Theorem \ref{Apropone} is to construct a suitable quasi-isomorphic model
for  the rows in the above diagram, making relation \eqref{eq:Aclassformula} 
manifest. In order to facilitate the exposition, diagrams such 
as \eqref{eq:AclassdiagCB} will be regarded in an obvious way as quiver representations in abelian categories of coherent sheaves. 
A map between two diagrams will be a morphism of quiver representations i.e. a collection of morphisms between the 
sheaves associated to the respective nodes which intertwine between the maps belonging to the two representations. 
Then, as a first step, note the following consequence of 
Corollary \ref{flatcorA}.

\begin{lem}\label{AclasslemC} 
Diagram 
\eqref{eq:AclassdiagCB} is canonically isomorphic to 
\be\label{eq:AclassdiagCC}
\xymatrix{ 
\CO_{\Delta_Z} \boxtimes \Delta_{T*}(\mathscr{J}_{T}/\mathscr{J}_{T}^2)
\ar[rr]^-{\nu_{Z_T}} \ar[d]^-{{\bf 1}} & &\mathcal{I}_{\Delta_{Z_U}}|_{{Z_T}\times {Z_T}} 
\ar[r]^{\iota_{Z_T} \circ \xi_{Z_T}} \ar[d]^-{r_{Z_T}}& \CO_{{Z_T}\times {Z_T}} \\ 
\CO_{\Delta_Z } \boxtimes \Delta_{T*}(\mathscr{J}_{T}/\mathscr{J}_{T}^2)
\ar[rr]^-{r_{Z_T}\circ \nu_{Z_T}} & & (\mathcal{I}_{\Delta_{Z_U}}/\mathcal{I}_{\Delta_{Z_U}}^2)|_{{Z_T}\times {Z_T}} & \\}
\ee
Moreover, the morphism $\nu_{Z_T}$ in diagram \eqref{eq:AclassdiagCC} 
is 
determined by the commutative diagram 
\be\label{eq:AclassdiagBC} 
\xymatrix{ 
\CO_{Z\times Z} \boxtimes (\CO_T \boxtimes \mathscr{J}_{T}) \ar@{^{(}->}[rr] \ar[d] 
& & \mathcal{I}_{\Delta_{Z_U}}|_{Z_T \times Z_U}\ar[d] \\
\CO_{\Delta_Z} \boxtimes \Delta_{T*}(\mathscr{J}_{T}/\mathscr{J}_{T}^2)
\ar[rr]^-{\nu_{Z_T}} & & \mathcal{I}_{\Delta_{Z_U}}|_{Z_T\times Z_T}.\\
}
\ee
where the top horizontal arrow is the natural injection induced 
by the isomorphism $\mathcal{I}_{\Delta_{Z_U}}|_{Z_T\times Z_U} \cong 
\mathcal{I}_{\Delta_{Z_T}\subset Z_T\times Z_U}$ and the vertical arrows are natural projections. 
\end{lem}
\begin{proof}
Using equation \eqref{eq:diagisomC} in 
Corollary \ref{flatcorA} it is straightforward to check 
that there is a commutative diagram of morphisms of 
$\CO_{Z_T\times Z_U}$-modules 
\be\label{eq:AclassdiagD} 
\xymatrix{
\CO_{Z_T} \boxtimes \mathscr{J}_{Z_T}  \ar[r]^-{\sim}\ar[d] &  \CO_{Z\times Z} \boxtimes (\CO_T\boxtimes \mathscr{J}_{T})\ar[d] \\
(\CO_{Z_T} \boxtimes \mathscr{J}_{Z_T})\otimes_{Z_T\times Z_U} \CO_{\Delta_{Z_T}} \ar[r]^-{\sim} & \CO_{\Delta_Z} \boxtimes \Delta_{T*} (\mathscr{J}_{T}/\mathscr{J}_{T}^2) \\}
\ee
where the horizontal arrows are canonical isomorphisms and the vertical arrows are natural projections. This follows from the basic properties of tensor products using the canonical isomorphisms 
\[
\mathscr{J}_{Z_T} \cong \CO_{Z} \boxtimes \mathscr{J}_{T}, \qquad (\CO_{T}\boxtimes \mathscr{J}_{T}) \otimes_{T\times U} \CO_{\Delta_T} \cong \Delta_{T*}\, (\mathscr{J}_{T}/\mathscr{J}_{T}^2)
\]
of $\CO_{Z_U}$ and $\CO_{T\times U}$-modules respectively. 
The details will be omitted. Moreover, as shown in Lemma \ref{AclasslemB}, the morphism $\nu_{Z_T}$ in diagram \eqref{eq:AclassdiagCB} is 
determined by the commutative diagram
\be\label{eq:AclassdiagBB}
\xymatrix{ 
\CO_{Z_T} \boxtimes \mathscr{J}_{Z_T} \ar@{^{(}->}[rr] \ar[d] 
& & \mathcal{I}_{\Delta_{Z_U}}|_{Z_T \times Z_U}\ar[d] \\
(\CO_{Z_T}\boxtimes \mathscr{J}_{Z_T})\otimes_{Z_T\times Z_U} 
\CO_{\Delta_{Z_T}}
\ar[rr]^-{\nu_{Z_T}} & & \mathcal{I}_{\Delta_{Z_U}}|_{Z_T\times Z_T} \\
}
\ee
where the top horizontal arrow is the natural injection while the vertical arrows are the natural projections. 
Diagram 
\eqref{eq:AclassdiagBC} follows immediately from 
\eqref{eq:AclassdiagBB} using \eqref{eq:AclassdiagD}.
\end{proof}

For the next step, recall the splitting \eqref{eq:Lsplitting} 
is determined by the canonical isomorphism 
\be\label{eq:idealsplit} 
{\overline \rho}_U: \mathcal{I}_{\Delta_{Z}}/\mathcal{I}_{\Delta_{Z}}^2 \boxtimes 
\CO_{\Delta_{U}} \oplus \CO_{\Delta_{Z}} \boxtimes 
\mathcal{I}_{\Delta_{U}}/\mathcal{I}_{\Delta_{U}}^2 {\buildrel \sim \over \longto} 
\mathcal{I}_{\Delta_{Z_U}}/\mathcal{I}_{\Delta_{Z_U}}^2
\ee
obtained in Corollary \ref{splitcor}. As also shown in loc. cit., there is a commutative diagram 
\be\label{eq:idealdiagC} 
\xymatrix{ 
\mathcal{I}_{\Delta_{Z}}/\mathcal{I}_{\Delta_{Z}}^2 \boxtimes 
\CO_{\Delta_{U}} \oplus \CO_{\Delta_{Z}} \boxtimes 
\mathcal{I}_{\Delta_{U}}/\mathcal{I}_{\Delta_{U}}^2 \ar[r]^-{{\overline \rho}_U} 
\ar[d]^-{(0, {\bf 1})} &
\mathcal{I}_{\Delta_{Z_U}}/\mathcal{I}_{\Delta_{Z_U}}^2\ar[d]^-{{\overline\gamma}_{Z_U}} \\
\CO_{\Delta_{Z}} \boxtimes 
\mathcal{I}_{\Delta_{U}}/\mathcal{I}_{\Delta_{U}}^2 \ar[r]^-{\bf 1} & 
\CO_{\Delta_{Z}} \boxtimes 
\mathcal{I}_{\Delta_{U}}/\mathcal{I}_{\Delta_{U}}^2 \\}
\ee
where ${\overline \gamma}_{Z_U}: \mathcal{I}_{\Delta_{Z_U}}/\mathcal{I}_{\Delta_{Z_U}}^2 \to 
\CO_{\Delta_{Z}} \boxtimes 
\mathcal{I}_{\Delta_{U}}/\mathcal{I}_{\Delta_{U}}^2$ is determined by the natural projection $\gamma_{Z_U}: \mathcal{I}_{\Delta_{Z_U}} \to \CO_{\Delta_{Z}} \boxtimes 
\mathcal{I}_{\Delta_{U}}$. 

Now let 
\[
p_{Z_T}={\overline \gamma}_{Z_U}|_{Z_T\times Z_T} :(\mathcal{I}_{\Delta_{Z_U}}/\mathcal{I}^2_{\Delta_{Z_U}})|_{Z_T \times Z_T} \to  \CO_{\Delta_Z}\boxtimes (\mathcal{I}_{\Delta_U}/\mathcal{I}_{\Delta_U}^2)|_{T\times T}. 
\]
and let $q_{Z_T}: \mathcal{I}_{\Delta_{Z_U}}|_{Z_T \times Z_T}\to 
\CO_{\Delta_Z}\boxtimes (\mathcal{I}_{\Delta_U}/\mathcal{I}_{\Delta_U}^2)|_{T\times T}$ be 
the composite map
\[
\xymatrix{
\mathcal{I}_{\Delta_{Z_U}}|_{Z_T \times Z_T} \ar[rr]^-{r_{Z_T}} 
\ar[drr]_-{q_{Z_T}} & &(\mathcal{I}_{\Delta_{Z_U}}/\mathcal{I}^2_{\Delta_{Z_U}})|_{Z_T \times Z_T}\ar[d]^-{p_{Z_T}}  \\
& & \CO_{\Delta_Z}\boxtimes (\mathcal{I}_{\Delta_U}/\mathcal{I}_{\Delta_U}^2)|_{T\times T}\, .}
\]
Here $r_{Z_T}$ is the natural projection used in diagram \eqref{eq:AclassdiagCC}. Then note that the canonical projection 
$\IL_{Z_T} \to \CO_Z \boxtimes \IL_T$ obtained from the splitting \eqref{eq:Lsplitting} is determined by the commutative diagram 
\[
\xymatrix{
\CO_{\Delta_Z } \boxtimes \Delta_{T*}(\mathscr{J}_{T}/\mathscr{J}_{T}^2) \ar[d]^-{\bf 1} 
\ar[rr]^-{r_{Z_T}\circ \nu_{Z_T}} & & 
(\CI_{\Delta_{Z_U}}/\CI^2_{\Delta_{Z_U}})|_{Z_T\times Z_T} 
\ar[d]^-{p_{Z_T}} \\ 
\CO_{\Delta_Z } \boxtimes \Delta_{T*}(\mathscr{J}_{T}/\mathscr{J}_{T}^2) \ar[rr]^-{q_{Z_T}\circ \nu_{Z_T}} & & \CO_{\Delta_Z} \boxtimes 
(\mathcal{I}_{\Delta_U}/\mathcal{I}_{\Delta_U}^2)|_{T\times T}.\\
}
\]
As a result, Lemma \ref{AclasslemC} immediately implies:
\begin{cor}\label{AclasslemD} 
The relative Atiyah class $\alpha_{Z_T/Z}$ is determined by the commutative diagram 
\be\label{eq:AclassdiagCD}
\xymatrix{ 
\CO_{\Delta_Z} \boxtimes \Delta_{T*}(\mathscr{J}_{T}/\mathscr{J}_{T}^2)
\ar[rr]^-{\nu_{Z_T}} \ar[d]^-{{\bf 1}} & &\mathcal{I}_{\Delta_{Z_U}}|_{{Z_T}\times {Z_T}} 
\ar[r]^{\iota_{Z_T} \circ \xi_{Z_T}} \ar[d]^-{q_{Z_T}}& \CO_{{Z_T}\times {Z_T}} \\ 
\CO_{\Delta_Z } \boxtimes \Delta_{T*}(\mathscr{J}_{T}/\mathscr{J}_{T}^2)
\ar[rr]^-{q_{Z_T}\circ \nu_{Z_T}} & & \CO_{\Delta_Z} \boxtimes 
(\mathcal{I}_{\Delta_U}/\mathcal{I}_{\Delta_U}^2)|_{T\times T}. & \\}
\ee
\end{cor} 

Note also that 
the Atiyah class of $T$ is obtained by setting $Z={\rm Spec}(\IC)$ in Corollary \ref{AclasscorA}. This yields the following commutative 
diagram 
\be\label{eq:AclassTA}
\xymatrix{ 
\Delta_{T*} (\mathscr{J}_{T}/\mathscr{J}_{T}^2)
 \ar[rrr]^-{\nu_{T}} \ar[d]^-{{\bf 1}} & & & \mathcal{I}_{\Delta_{U}}|_{T\times T} 
\ar[r]^{\iota_{T} \circ \xi_{T}} \ar[d]^-{r_{T}}& \CO_{T\times T} \\ 
\Delta_{T*} (\mathscr{J}_{T}/\mathscr{J}_{T}^2)
\ar[rrr]^-{r_{T}\circ \nu_{T}} & & &  
\mathcal{I}_{\Delta_U}/\mathcal{I}^2_{\Delta_U}|_{T\times T}. & \\}
\ee
where the bottom row is isomorphic to $\IL_T[1]$. 
The map $\nu_T$ is determined by the commutative diagram 
\be\label{eq:AclassTB} 
\xymatrix{ 
\CO_T \boxtimes \mathscr{J}_{T} \ar@{^{(}->}[rr] \ar[d] 
& & \mathcal{I}_{\Delta_{U}}|_{T \times U}\ar[d] \\
\Delta_{T*}(\mathscr{J}_{T}/\mathscr{J}_{T}^2)
\ar[rr]^-{\nu_{T}} & & \mathcal{I}_{\Delta_{U}}|_{T\times T}.\\
}
\ee
where the top horizontal map is the natural inclusion induced by the isomorphism $\mathcal{I}_{\Delta_U}|_{T \times U} \cong \mathcal{I}_{\Delta_T\subset T\times U}$, and the vertical 
maps are natural projections. 
 
 As the next step in the proof of Theorem \ref{Apropone}, 
 it will be shown below in Lemma \ref{diaglemBC} that the domain of the relative Atiyah class is quasi-isomorphic to the three term complex 
\be\label{eq:qisomdom} 
\xymatrix{ 
 \CO_{\Delta_Z} \boxtimes \Delta_{T*}(\mathscr{J}_{T}/\mathscr{J}^{2}_{T})  
 \ar[r]^-{{\bf 1}\boxtimes \nu_T}  & \CO_{\Delta_Z}\boxtimes \mathcal{I}_{\Delta_U}|_{T\times T} 
 \ar[rr]^-{{\bf 1}\boxtimes (\iota_T\circ \xi_T)} & & \CO_{\Delta_Z} \boxtimes \CO_{T\times T}.  & \\
}
\ee
This is an important step in the proof since it provides a direct 
connection with the domain of the $\alpha_T$ in diagram 
\eqref{eq:AclassTB}. 

The starting point is the exact sequence of $\CO_{Z_U \times Z_U}$-modules
\be\label{eq:idealseqC} 
0 \to \mathcal{I}_{\Delta_{Z}}\boxtimes \CO_{U\times U} 
\xlongto{j_{Z_U}}
\mathcal{I}_{\Delta_{Z_U}} \xlongto{\gamma_{Z_U}}   \CO_{\Delta_{Z}} \boxtimes \mathcal{I}_{\Delta_{U}} \to 0,
\ee
which follows from 
Corollary \ref{flatcorB}. The key idea is to gain some control over the restriction $\mathcal{I}_{\Delta_{Z_U}}|_{Z_T \times Z_T}$ by 
proving: 

\begin{lem}\label{diaglemB} 
The exact sequence \eqref{eq:idealseqC} 
remains exact 
upon restriction to $Z_T\times Z_T$.
\end{lem} 

\begin{proof}
First note that Corollary \ref{flatcorB} yields a second 
exact sequence of $\CO_{Z_U\times{Z_U}}$-modules 
\be\label{eq:idealseqD}
0 \to \mathcal{I}_{\Delta_{Z}}\boxtimes \mathcal{I}_{\Delta_U} \xlongto{\delta_{Z_U}}
\mathcal{I}_{\Delta_{Z}}\boxplus \mathcal{I}_{\Delta_{U}} \xlongto{\rho_{Z_U}}   \mathcal{I}_{\Delta_{Z_U}} \to 0.
\ee
As shown below, this sequence remains exact upon restriction to $Z_T \times Z_T$. The restriction of \eqref{eq:idealseqD} to $Z_T \times Z_T$ 
is the exact sequence 
\be\label{eq:IgammaUseqA} 
\mathcal{I}_{\Delta_Z} \boxtimes (\mathcal{I}_{\Delta_U}|_{T\times T}) \to \mathcal{I}_{\Delta_Z} \boxplus (\mathcal{I}_{\Delta_U}|_{T\times T}) \to 
I_{\Delta_{Z_U}}|_{Z_T\times Z_T}  \to 0.
\ee
The second component of the first map from the left is 
\be\label{eq:secondcomp}
\iota_Z\boxtimes {\bf 1}: \mathcal{I}_{\Delta_Z} \boxtimes \mathcal{I}_{\Delta_U}|_{T\times T} \to \CO_{Z\times Z}\boxtimes \mathcal{I}_{\Delta_U}|_{T\times T}
\ee
where $\iota_Z: \mathcal{I}_{\Delta_Z} \to \CO_{Z\times Z}$ is the canonical injection. This is injective, as shown in Lemma \ref{flatlemma} hence \eqref{eq:IgammaUseqA} is 
also left exact. 

Now, the restriction of \eqref{eq:idealseqC} to $Z_T\times Z_T$ is the  exact sequence 
\be\label{eq:ZTseq}
\mathcal{I}_{\Delta_Z} \boxtimes \CO_{T\times T} {\buildrel \phi \over \longto}   
\mathcal{I}_{\Delta_{Z_U}}|_{T\times T}  {\buildrel \over \longto} \CO_{\Delta_Z} \boxtimes \mathcal{I}_{\Delta_U}|_{T\times T} \to 0. 
\ee
where $\phi= j_{Z_U}|_{Z_T\times Z_T}$. 
Furthermore, 
by construction there is a commutative diagram  
\[ 
\xymatrix{ 
& & \mathcal{I}_{\Delta_Z}\boxtimes \CO_{T\times T} \ar[r]^-{\bf 1} 
\ar[dd]_-{\left(\begin{array}{c} {\bf 1} \\ 0 \end{array}\right)} 
& \mathcal{I}_{\Delta_Z}\boxtimes \CO_{T\times T} \ar[dd]^-\phi \\
& & & \\
0 \ar[r] & \mathcal{I}_{\Delta_Z} \boxtimes (\mathcal{I}_{\Delta_U}|_{T\times T}) \ar[r] & \mathcal{I}_{\Delta_Z} \boxplus (\mathcal{I}_{\Delta_U}|_{T\times T}) \ar[r] & 
I_{{\Delta_Z}_U}|_{T\times T}  \ar[r] &  0.\\}
\]
where the bottom row is the restriction of \eqref{eq:idealseqD}
to $Z_T \times Z_T$.  Since the latter is exact, using equation \eqref{eq:secondcomp},  the snake lemma yields a long exact sequence 
\[ 
0 \to {\rm Ker}(\phi) \to \mathcal{I}_{\Delta_Z} \boxtimes (\mathcal{I}_{\Delta_U}|_{T\times T}) \xlongto{\iota_Z\boxtimes {\bf 1}} \CO_{Z\times Z} \boxtimes (\mathcal{I}_{\Delta_U}|_{T\times T}) \to \cdots. 
\]
As noted above, $\iota_Z\boxtimes {\bf 1}$ is injective, hence 
\eqref{eq:ZTseq} is also exact on the left. 
\end{proof}

Now recall that the domain of the Atiyah class in diagram \eqref{eq:AclassdiagCD} is the complex 
\[
\CO_{\Delta_Z} \boxtimes \Delta_{T*} (\mathscr{J}_{T}/\mathscr{J}_{T}^2) \xlongto{\nu_{Z_T}} 
\mathcal{I}_{\Delta_{Z_U}}|_{Z_T\times Z_T} 
\xlongto{\iota_{Z_T}\circ \xi_{Z_T}} 
\CO_{Z\times Z} \boxtimes \CO_{T\times T}.
\]
Let 
\be\label{eq:idealseqE} 
0 \to \mathcal{I}_{\Delta_Z}\boxtimes \CO_{T\times T} 
\xlongto{\epsilon_{Z_T}} 
\mathcal{I}_{\Delta_{Z_U}}|_{Z_T\times Z_T} \xlongto{\eta_{Z_T}} 
\CO_{\Delta_{Z}} \boxtimes \mathcal{I}_{\Delta_U}|_{T\times T} \to 0 
\ee
denote the restriction of \eqref{eq:idealseqC} to $Z_T \times Z_T$ and let  
\be\label{eq:candiagseqA} 
0 \to \mathcal{I}_{\Delta_Z} {\buildrel \iota_Z\over \longto} 
\CO_{Z\times Z} 
{\buildrel f_Z \over \longto} \CO_{\Delta_Z}\to 0.
\ee
be the canonical exact sequence of the diagonal in $Z\times Z$.

As claimed above equation \eqref{eq:qisomdom}, the next result proves that the domain of the relative Atiyah class is indeed quasi-isomorphic to the complex \eqref{eq:qisomdom}.
\begin{lem}\label{diaglemBC} 
There is a commutative diagram 
\be\label{eq:bigdiagA} 
\xymatrix{ 
&  0\ar[d] & & 0 \ar[d] &  \\
 & \mathcal{I}_{\Delta_Z} \boxtimes \CO_{T\times T}  \ar[d]_-{\epsilon_{Z_T}} 
\ar[rr]^-{ {\bf 1}} & & \mathcal{I}_{\Delta_Z} \boxtimes \CO_{T\times T}
\ar[d]_-{i_Z\boxtimes {\bf 1}} 
 &  \\
 \CO_{\Delta_Z} \boxtimes \Delta_{T*}\mathscr{J}_{T}/\mathscr{J}^{2}_{T} \ar[d]_-{\bf 1}
 \ar[r]^-{\nu_{Z_T}} & 
\mathcal{I}_{\Delta_{Z_U}}|_{Z_T\times Z_T}  \ar[d]_-{\eta_{Z_T}}
\ar[rr]^-{\iota_{Z_T}\circ \xi_{Z_T}}  & & \CO_{Z\times Z}\boxtimes \CO_{T\times T}\ar[d]_-{f_Z\boxtimes {\bf 1}} \\ 
 \CO_{\Delta_Z} \boxtimes \Delta_{T*}(\mathscr{J}_{T}/\mathscr{J}^{2}_{T})  
 \ar[r]^-{{\bf 1}\boxtimes \nu_T}  & \CO_{\Delta_Z}\boxtimes \mathcal{I}_{\Delta_U}|_{T\times T} 
 \ar[rr]^-{{\bf 1}\boxtimes (\iota_T\circ \xi_T)}\ar[d] & & \CO_{\Delta_Z} \boxtimes \CO_{T\times T}  \ar[d] & \\
 & 0 & & 0 &  \\
 } 
 \ee
with exact columns. 
\end{lem} 

\begin{proof}
The middle column is the exact sequence \eqref{eq:idealseqE}. 
The exactness of the right column follows from the exactness of \eqref{eq:candiagseqA} using Lemma \ref{flatlemma}. The lower right 
square is commutative since all the arrows in it are natural restrictions. 

In order to prove that the lower left square is commutative, 
note that the maps $\nu_{Z_T},\nu_T$ are determined by the commutative diagrams \eqref{eq:AclassdiagD} and \eqref{eq:AclassTB} respectively. 
These are displayed again below for convenience:
\be\label{eq:nudiag}
\xymatrix{
\CO_{Z\times Z} \boxtimes (\CO_T\boxtimes \mathscr{J}_{T}) 
\ar@{^{(}->}[rr]^-{k_{Z_T}}  \ar[d]^-{a_{Z_T}} 
& & \mathcal{I}_{\Delta_{Z_U}}|_{Z_T\times Z_U}\ar[d]^-{b_{Z_T}} \\
\CO_{\Delta_Z} \boxtimes \Delta_{T*} (\mathscr{J}_{T}/\mathscr{J}_{T}^2)  
\ar[rr]^-{\nu_{Z_T}} & & \mathcal{I}_{\Delta_{Z_U}}|_{Z_T\times Z_T} 
 \\
}
\ee
\[
\xymatrix{
\CO_T\boxtimes \mathscr{J}_{T} \ar@{^{(}->}[rr]^-{k_T} \ar[d]^-{a_T}
& & \mathcal{I}_{\Delta_{U}}|_{T\times U}\ar[d]^-{b_T} \\
\Delta_{T*} (\mathscr{J}_{T}/\mathscr{J}_{T}^2)  
\ar[rr]^-{\nu_{T}} & & \mathcal{I}_{\Delta_{U}}|_{T\times T} 
 \\
}
\]
 In each case the vertical arrows are canonical projections while the upper horizontal arrow is a natural injection. Obviously, the second` diagram in \eqref{eq:nudiag} yields a 
commutative diagram of $\CO_{Z_T \times Z_U}$-modules 
\be\label{nudiag2}
\xymatrix{
\CO_{\Delta_Z}\boxtimes (\CO_T\boxtimes \mathscr{J}_{T})
\ar@{^{(}->}[rr]^-{{\bf 1}\boxtimes k_T} \ar[d]^-{{\bf 1}\boxtimes a_T}
& & \CO_{\Delta_Z}\boxtimes(\mathcal{I}_{\Delta_{U}}|_{T\times U})\ar[d]^-{{\bf 1}\boxtimes b_T} \\
\CO_{\Delta_Z}\boxtimes\Delta_{T*} (\mathscr{J}_{T}/\mathscr{J}_{T}^2)  
\ar[rr]^-{{\bf 1}\boxtimes \nu_{T}} & & \CO_{\Delta_Z}\boxtimes
(\mathcal{I}_{\Delta_{U}}|_{T\times T}) 
 \\
}
\ee
Then the from Diagram \eqref{nudiag2} and the left diagram in \eqref{eq:nudiag} one obtains a diagram of the form 
\be\label{eq:polydiagB}
\xymatrix@R=1.5em{
\CO_{Z\times Z} \boxtimes (\CO_T\boxtimes \mathscr{J}_{T}) 
\ar@{^{(}->}[rr]^-{k_{Z_T}}  \ar[d]^-{a_{Z_T}} 
\ar@/_5.5pc/[ddd]_(.3){f_Z\boxtimes {\bf 1}}|\hole& & \mathcal{I}_{\Delta_{Z_U}}|_{Z_T\times Z_U}\ar[d]^-{b_{Z_T}} 
\ar@/^4pc/[ddd]^(.3){\gamma_{Z_U}|_{Z_T\times Z_U}}|\hole\\
\CO_{\Delta_Z} \boxtimes \Delta_{T*} (\mathscr{J}_{T}/\mathscr{J}_{T}^2)  
\ar@/_11pc/[ddd]_-{\bf 1}
\ar[rr]^-{\nu_{Z_T}} & & \mathcal{I}_{\Delta_{Z_U}}|_{Z_T\times Z_T} 
\ar@/^9pc/^-{\eta_{Z_T}}[ddd] \\
 & & \\
\CO_{\Delta_Z}\boxtimes (\CO_T\boxtimes \mathscr{J}_{T})
\ar@{^{(}->}[rr]^-{{\bf 1}\boxtimes k_T} \ar[d]^-{{\bf 1}\boxtimes a_T}
& & \CO_{\Delta_Z}\boxtimes(\mathcal{I}_{\Delta_{U}}|_{T\times U})\ar[d]^-{{\bf 1}\boxtimes b_T} \\
\CO_{\Delta_Z}\boxtimes\Delta_{T*} (\mathscr{J}_{T}/\mathscr{J}_{T}^2)  
\ar[rr]^-{{\bf 1}\boxtimes \nu_{T}} & & \CO_{\Delta_Z}\boxtimes
(\mathcal{I}_{\Delta_{U}}|_{T\times T}) 
 \\}
 \ee
where $\gamma_U: \mathcal{I}_{\Delta_{Z_U}} \to \CO_{\Delta_Z} \boxtimes \mathcal{I}_{\Delta_U}$  is the projection in \eqref{eq:idealseqC}. 
Since $k_{Z_T}, k_T$ are natural injections, the following 
relations hold
\[
\bal
\gamma_{Z_U}|_{Z_T \times Z_U} \circ k_{Z_T} & = ({\bf 1}\boxtimes k_T)\circ (f_Z\boxtimes {\bf 1})\\
\eal 
\]
Moreover, since $\eta_{Z_T} = \gamma_{U}|_{Z_T\times Z_T}$
and $b_{Z_T}, b_T$ are  natural projections, one also has 
\[
\bal 
(1\boxtimes b_T) \circ \gamma_U|_{Z_U \times Z_T} & = 
\eta_{Z_T} \circ b_{Z_T}. \\
\eal
\]
Since $a_T, a_{Z_T}$ are surjections, and two straight squares 
in the diagram \eqref{eq:polydiagB} are commutative, this implies 
\[
\eta_{Z_T} \circ \nu_{Z_T} =({\bf 1}\boxtimes \nu_T).
\]

Finally, in order to prove that the upper right square is commutative, 
first note that Lemma \ref{flatcorB} yields an exact sequence 
\[
0 \to \mathcal{I}_{\Delta_Z}\boxtimes \CO_{T\times T} 
\xlongto{\delta_{Z_T}} 
\mathcal{I}_{\Delta_{Z_T}} \xlongto{\gamma_{Z_T}} 
\CO_{\Delta_Z}\boxtimes \mathcal{I}_{\Delta_T} \to 0.
\]
Then
note the diagram 
\be\label{eq:polydiagA}
\xymatrix{ 
\mathcal{I}_{\Delta_Z} \boxtimes \CO_{T\times T}  \ar[d]_-{\epsilon_{Z_T}} 
\ar[r]^-{ {\bf 1}} & \mathcal{I}_{\Delta_Z} \boxtimes \CO_{T\times T}
\ar[d]_-{\delta_{Z_T}}  \ar[ddr]^-{\iota_Z \boxtimes {\bf 1}} & \\
\mathcal{I}_{\Delta_{Z_U}}|_{Z_T \times Z_T} \ar[r]^-{\xi_{Z_T}} 
\ar[drr]_-{\iota_{Z_T}\circ \xi_{Z_T}}  & \mathcal{I}_{\Delta_{Z_T}} \ar[dr]^-{\iota_{Z_T}} & \\
& & \CO_{Z\times Z} \boxtimes \CO_{T\times T} 
\\}
\ee
Since all the maps involved are natural injections or projections, it is straightforward to check that the square and the triangles are commutative. This implies that the outer polygon is also commutative. 
\end{proof}

Finally, the next lemma provides a relation between the target of the relative Atiyah class $\alpha_{Z_T/Z}$ and the target of 
$\alpha_T$. 
\begin{lem}\label{diaglemC} 
There is a further commutative diagram 
\be\label{eq:bigdiagC}
\xymatrix@C=1em @R=1.5em{
\CO_{\Delta_Z} \boxtimes \Delta_{T*}(\mathscr{J}_{T}/\mathscr{J}_{T}^2) \ar[rr]^-{\nu_{Z_T}}  \ar[d]^-{{\bf 1}} 
\ar@/_6pc/[ddd]_(.3){\bf 1}|\hole & & \mathcal{I}_{\Delta_{Z_U}}|_{Z_T\times Z_T}\ar[d]^-{r_{Z_T}} 
\ar@/^5.5pc/[ddd]^(.3){\eta_{Z_T}}|\hole\\
\CO_{\Delta_Z} \boxtimes \Delta_{T*} (\mathscr{J}_{T}/\mathscr{J}_{T}^2)  
\ar@/_12pc/[ddd]_-{\bf 1}
\ar[rr]^-{q_{Z_T}\circ \nu_{Z_T}} & & \CO_{\Delta_Z}\boxtimes (\mathcal{I}_{\Delta_{U}}/I^2_{\Delta_{U}})|_{T\times T} 
\ar@/^11pc/^-{{\bf 1}}[ddd] \\
 & & \\
\CO_{\Delta_Z}\boxtimes \Delta_{T*} (\mathscr{J}_{T}/\mathscr{J}_{T}^2)
\ar[rr]^-{{\bf 1}\boxtimes \nu_T} \ar[d]^-{{\bf 1}}
& & \CO_{\Delta_Z}\boxtimes \mathcal{I}_{\Delta_{U}}|_{T\times T}
\ar[d]^-{{\bf 1}\boxtimes r_T} \\
\CO_{\Delta_Z}\boxtimes\Delta_{T*} (\mathscr{J}_{T}/\mathscr{J}_{T}^2)  
\ar[rr]^-{{\bf 1}\boxtimes s_T} 
& & \CO_{\Delta_Z}\boxtimes
(\mathcal{I}_{\Delta_{U}}/I^2_{\Delta_U})|_{T\times T}
 \\}
\ee
where $s_T = r_T \circ \nu_T$, and $\eta_{Z_T}$ is the third map from the left in the exact 
sequence \eqref{eq:idealseqE}. 
In particular, the following relation holds:
\be\label{eq:qcommrel}
q_{Z_T}\circ \nu_{Z_T} = {\bf 1}\boxtimes(r_T\circ  \nu_T). 
\ee
 \end{lem}
 
 \begin{proof}The relation 
 $\eta_{Z_T}\circ \nu_{Z_T} = {\bf 1}\boxtimes \nu_T$ has already 
 been proven in Lemma \ref{diaglemB}. 
 The rest follows from the fact that the $r_{Z_T}, r_T$
 and  
 $\eta_{Z_T}$ are natural projections. 
 
 \end{proof}

\subsection{Proof of Theorem \ref{Apropone}}  
To summarize, Lemma \ref{AclasslemC} shows that the relative Atiyah 
class $\alpha_{Z_T/Z}$ is determined by the commutative diagram 
\[ 
\xymatrix{ 
\CO_{\Delta_Z} \boxtimes \Delta_{T*}(\mathscr{J}_{T}/\mathscr{J}_{T}^2)
\ar[rr]^-{\nu_{Z_T}} \ar[d]^-{{\bf 1}} & &\mathcal{I}_{\Delta_{Z_U}}|_{{Z_T}\times {Z_T}} 
\ar[r]^{\iota_{Z_T} \circ \xi_{Z_T}} \ar[d]^-{q_{Z_T}}& \CO_{{Z_T}\times {Z_T}} \\ 
\CO_{\Delta_Z } \boxtimes \Delta_{T*}(\mathscr{J}_{T}/\mathscr{J}_{T}^2)
\ar[rr]^-{q_{Z_T}\circ \nu_{Z_T}} & & \CO_{\Delta_Z } \boxtimes (\mathcal{I}_{\Delta_{U}}/\mathcal{I}_{\Delta_{U}}^2)|_{{Z_T}\times {Z_T}} & \\}
\]
Let $\CA_{Z_T}$ and $\CC_{Z_T}$ denote the top and the bottom row respectively.
At the same time the Atiyah class of $T$ is determined by the 
diagram 
\[ 
\xymatrix{ 
\Delta_{T*} (\mathscr{J}_{T}/\mathscr{J}_{T}^2)
 \ar[rrr]^-{\nu_{T}} \ar[d]^-{{\bf 1}} & & & \mathcal{I}_{\Delta_{U}}|_{T\times T} 
\ar[r]^{\iota_{T} \circ \xi_{T}} \ar[d]^-{r_{T}}& \CO_{T\times T} \\ 
\Delta_{T*} (\mathscr{J}_{T}/\mathscr{J}_{T}^2)
\ar[rrr]^-{r_{T}\circ \nu_{T}} & & &  
\mathcal{I}_{\Delta_U}/I^2_{\Delta_U}|_{T\times T}. & \\}
\]
Again, let $\CA_{T}$ and $\CC_{T}$ denote the top and the bottom row respectively. 

Then Lemmas \ref{diaglemB} and \ref{diaglemC} show that there is a commutative diagram of complexes 
\[
\xymatrix{ 
\CA_{Z_T}\ar[d]^-{\alpha_{Z_T/Z}} \ar[r]^-{\sim} & 
\CO_{\Delta_Z} \boxtimes \CA_T \ar[d]^-{{\bf 1}\boxtimes \alpha_T}
\\
\CC_{Z_T} \ar[r]^-{\sim} & \CO_{\Delta_Z}\boxtimes \CC_T \\ }
\]
where the top horizontal arrow is a quasi-isomorphism and the bottom row is a canonical isomorphism. This proves the claim in Theorem \ref{Apropone}. 
\hfill $\text{q.e.d.}$

\section{Obstruction theory for sheaves on local Calabi-Yau varieties}\label{obstrsect}

This section serves as an interesting application of Theorem \ref{Apropone}, providing the foundations for a sheaf counting theory of local Calabi-Yau fourfolds. To begin with, we start by discussing the particular moduli theory we are interested in this article. Let $Y$ be a smooth quasi-projective variety and let $X$ be the total space of the canonical bundle $K_Y$. Let $i:Y \to X$ denote the zero section. 
Let 
$\CM$ be a quasi-projective fine moduli scheme of stable sheaves, $\mathscr{F}$, on $X$ with proper support. 
Suppose all sheaves $\mathscr{F}$ parameterized by $\CM$ are scheme-theoretically supported on $Y$.
Then the goal of this section is to prove that 
the obstruction theory of $\CM$ as a moduli space of sheaves on $X$ coincides with the obstruction theory of $\CM$ as a moduli space of sheaves on $Y$. The precise statement is Proposition \ref{obsprop}, which relies on Proposition \ref{Aclasslocal} and Theorem \ref{Apropone}. The obstruction theory will be defined as in the work of Behrend and Fantechi \cite{BF97} and the proof will use the formalism of truncated Atiyah classes, which was discussed in the previous section. 

In order to facilitate the exposition, the truncated Atiyah class will be referred to simply as the Atiyah class below
since no other variant will be used throughout the paper. To fix notation, let $\CY= Y \times \CM$, $\CX= X \times \CM$, and let us denote the canonical embedding $i_\CY := i \times 1_\CM:\CY \to \CX$. 
Let $\IF$ denote the universal sheaf on $\CX$ \footnote{The moduli space $\CM$ is in fact the coarse moduli 
space of an Artin moduli stack of ${\mathfrak M}$ of stable sheaves. 
The latter is a $\IC^*$-gerbe over $\CM$ since all stable objects have $\IC^*$-stabilizers, and hence a universal sheaf $\IF$ might not exist on $X\times \CM$ globally. As explained in \cite{T98}, 
below Theorem 3.30, this issue can be ignored, as such a sheaf $\IF_i$ exists locally on open subsets $X \times \CU_i\subset X\times \CM$ and different local choices for $\IF_i$ differ by twisting by line bundles $\mathcal{L}_{i}$ pulled back from $U_{i}$. However the direct image 
$\dR{\mathscr{H}om}_{\pi_{\CU_i}}(\IF_i,\IF_i)$ used in our construction remains invariant under twisting by $\mathcal{L}_{i}$, and hence there is a unique global object 
$\dR{\mathscr{H}om_{\pi_\CM}(\IF,\IF)}$ on  $X\times \CM$, independent of any choices.}. Under the current assumptions, $\IF$ is the extension by zero, $\IF = i_{\CY*}\IG$, of a sheaf $\IG$ on $\CY$. The latter is also a universal sheaf on 
$\CY$. Following the framework of Section \ref{Aclassproduct}, as shown in \cite{HT10}, the relative Atiyah class 
$\alpha_{\CX/X}$ of $\CX$, is given by the morphism in the derived category
\be\label{eq:relAclassM}
\alpha_{\CX/X}: \CO_{\Delta_\CX} \to \Delta_{\CX*}(\pi_\CM^*\IL^{\bullet}_\CM[1]),
\ee
which determines a natural transformation between the Fourier-Mukai functors $\mathscr{D}^b(\CX)\to \mathscr{D}^b(\CX)$ with kernels $\CO_{\Delta_\CX}$ and $\Delta_{\CX*}(\pi_\CM^*\IL^{\bullet}_\CM[1])$ respectively. Therefore, given a sheaf $\IF$ on $\CX$, 
the map \eqref{eq:relAclassM} determines, by Fourier-Mukai transform, a morphism 
\[
\alpha_{\CX/X}(\IF) : \IF \to \IF \otimes^L  \pi_\CM^*\IL^{\bullet}_\CM[1].
\]
This is the universal relative Atiyah class of $\IF$.

\subsection{Comparison of relative Atiyah 
classes}\label{compAclass} 

Now let $\pi_X: \CX\to X$, 
$\pi_\CM: \CX\to \CM$, $\rho_Y: \CY\to Y$
and $\rho_\CM: \CY\to \CM$  denote the canonical projections. 

\begin{notn} For any scheme $W$ over $\IC$ the bounded derived category of $W$ will be denoted by  $\mathscr{D}^b(W)$. 
For any two objects $C_1,C_2$ of $\mathscr{D}^b(W)$ let 
$\dR{\rm Hom}_{\mathscr{D}^b(W)}(C_1,C_2)$, and 
$\dR{\mathscr Hom}_{\mathscr{D}^b(W)}(C_1,C_2)$ denote global and local derived Hom functors respectively. Moreover, for a morphism of schemes 
$f: W \to V$ let $\dR{\mathscr Hom}_f(C_1,C_2)$ denote the derived pushforward 
${\dR}f_*\dR{\mathscr Hom}_{\mathscr{D}^b(W)}(C_1,C_2)$. 
\end{notn}

The main goal of this section is to compare the relative Atiyah classes of $\IF$ and $\IG$,
\[ 
\bal 
\alpha_{\CX/X}(\IF) & \in \dR{\rm {H}om}_{\mathscr{D}^b(\CX)}(\IF, \IF\otimes^L 
\pi_\CM^*\IL^{\bullet}_\CM)[1],\\ 
 \alpha_{\CY/Y}(\IG) &  \in \dR{\rm {H}om}_{\mathscr{D}^b(\CY)}(\IG, \IG\otimes^L 
\rho_\CM^*\IL^{\bullet}_\CM)[1],\\
\eal 
\]
defined in Section \ref{Aclassproduct}. 
First, note the following relations in the derived category of
$\CX$. 

\begin{lem}\label{adjlemB} 
Let $Q$ be an object of $\mathscr{D}^b(\CM)$. Then 
there is a natural isomorphism 
\be\label{eq:extisomA} 
\bal 
\dR{\mathscr{H}om}_{\mathscr{D}^b(\CX)} (\IF, \IF\otimes^L 
\pi_\CM^*Q)[1] \cong\  & i_{\CY*}\dR{\mathscr{H}om}_{\mathscr{D}^b(\CY)}(\IG, \IG \otimes^L \rho_\CM^*Q)[1] \oplus \\
& i_{\CY*}\dR{\mathscr{H}om}_{\mathscr{D}^b(\CY)}(\IG\otimes \rho_Y^*K^{-1}_Y, \IG \otimes^L \rho_\CM^*Q) \\
\eal
\ee
\end{lem} 

\begin{proof}
The projection formula 
\cite[Prop. 5.6, Ch. II.5]{H66} 
yields an isomorphism 
\[
\IF\otimes^L \pi_{\CM}^*Q
\cong i_{\CY*}(\IG\otimes^L Li_\CY^*\pi_{\CM}^*Q), 
\]
where $i_{\CY*}\CO_\CY\cong \pi_X^* i_*\CO_Y$, hence Lemma 
\ref{flatlemma} implies that 
\be\label{eq:dpbid} 
Li_\CY^* \pi_\CM^*Q \cong i_\CY^*\pi_\CM^* Q \cong \rho_\CM^*Q
\ee 
One is then left with an isomorphism 
\[
\IF\otimes^L \pi_{\CM}^*Q
\cong i_{\CY*}(\IG\otimes^L \rho_\CM^*Q),
\]
and the claim follows from Lemma \ref{adjlem}. 

\end{proof}
In particular, setting $Q=\IL^{\bullet}_\CM$ in Lemma \ref{adjlemB}, and applying the derived functor $\dR\Gamma_{{\mathscr D}^b(\CX)}$ to both sides of equation 
\eqref{eq:extisomA}, 
one obtains a decomposition 
\be\label{eq:extisomB} 
\bal 
\dR{\rm {H}om}_{\mathscr{D}^b(\CX)} (\IF, & \IF\otimes^L 
\pi_\CM^*\IL^{\bullet}_\CM)[1] \cong\  \\
&  \dR{\rm {H}om}_{\mathscr{D}^b(\CY)}(\IG, \IG \otimes^L \rho_\CM^*\IL^{\bullet}_\CM)[1]\ \oplus \\
& \dR{\rm {H}om}_{\mathscr{D}^b(\CY)}(\IG\otimes \CK^{-1}, \IG \otimes^L \rho_\CM^*\IL^{\bullet}_\CM), \\
\eal
\ee
where $\CK = \rho_Y^*K_Y$.
The first main goal of this section is to prove: 

\begin{prop}\label{Aclasslocal} 
The components of $\alpha_{\CX/X}(\IF)$ with respect to the direct sum decomposition \eqref{eq:extisomB} are 
\be\label{eq:Aclassadj} 
( \alpha_{\CY/Y}(\IG),\  0).
\ee
\end{prop} 

\begin{rmk}A similar result has been proven by Kuznetsov and Markusevich in \cite[Theorem 3.2]{kuznets}. For completeness we will recall their theorem in here.

\begin{thr}\cite[Theorem 3.2]{kuznets}
Let $i: Y\to M$ be a locally complete intersection.
\begin{enumerate}[label=(\roman*)]
\item For any $\mathscr{F}\in \mathscr{D}^b(\operatorname{Coh}(Y))$, the linkage class $\epsilon_{\mathscr{F}}\in \operatorname{Ext}^{2}(\mathscr{F}, \mathscr{F}\otimes \mathscr{N}^{\vee}_{Y/M})$ is the product of the Atiyah class $\operatorname{At}_{\mathscr{F}}\in \operatorname{Ext}^{1}(\mathscr{F}, \mathscr{F}\otimes \Omega_{Y})$ with $\nu_{Y/M}\in \operatorname{Ext}^{1}(\Omega_{Y}, \mathscr{N}_{Y/M}^{\vee})$. In other words $\epsilon_{\mathscr{F}}=\left(1_{\mathscr{F}\otimes \nu_{Y/M}}\right)\circ \operatorname{At}_{\mathscr{F}}$.
\item For any $\mathscr{G}\in \mathscr{D}^b(\operatorname{Coh}(M))$ we have $$\operatorname{At}_{Li^*\mathscr{G}}\cong \rho_{*}\left((\operatorname{At}_{\mathscr{G}})|_{Y}\right),$$where $\rho_{*}: \operatorname{Ext}^{1}(Li^*\mathscr{G}, Li^*\mathscr{G}\otimes \Omega_{M}|_{Y})\to \operatorname{Ext}^{1}(Li^*\mathscr{G},Li^*\mathscr{G}\otimes \Omega_{Y})$ is the pushout via $\rho: \Omega_{M}|_{Y}\to \Omega_{Y}$.
\item For any $\mathscr{F}\in \mathscr{D}^b(\operatorname{Cor}(Y))$ the image of the Atiyah class $\operatorname{At}_{i_{*}\mathscr{F}}\in \operatorname{Ext}^{1}(i_{*}\mathscr{F},i_{*}\mathscr{F}\otimes \Omega_{M})$ in $\dR\operatorname{Hom}^{1}(Li^*i_{*}\mathscr{F},i_{*}\mathscr{F}\otimes \Omega_{M})=H^{0}(M,i_{*}(\mathscr{F}^{\vee})\otimes \mathscr{F}\otimes \mathcal{N}_{Y/M}\otimes \Omega_{M}|_{Y})$ equals  $1_{\mathscr{F}}\otimes k$ where $k=k_{Y/M}:\mathscr{N}^{\vee}_{Y/M}\to \Omega_{M}|_{Y}$. 
\end{enumerate}
\end{thr}
Note that the above theorem is a point-wise statement which uses the classical Atiyah classes instead of universal Atiyah class over the moduli space. In this section we aim at providing an analogous proof of part (iii), with the relative truncated Atiyah class $\alpha_{\CX/X}$ in place of the classical Atiyah class. The authors believe that lifting the proof of Proposition \ref{Aclasslocal} to the one using the full un-truncated cotangent complex, is possible although, it might be a big step involving  difficult technicalities.
\end{rmk}

As in \cite[Theorem 3.2]{kuznets}, the proof will use the following commutative diagram, where both squares are cartesian:
\be\label{eq:XYsquares} 
\xymatrix{ 
\CY \ar[d]^-{\Delta_\CY} \ar[rr]^-{1} & &  \CY \ar[d]^-{\Gamma_\CY} \ar[rr]^-{i_\CY} & & \CX \ar[d]^-{\Delta_\CX} 
\\
\CY\times \CY \ar[rr]^-{i_\CY \times 1} & & \CX \times \CY\ar[rr]^-{1\times i_\CY} & & \CX \times \CX. \\}
\ee
As shown in detail below, the successive application 
of the projection formula for the maps $1\times i_\CY$ and 
$i_\CY \times 1$ will yield Equation \eqref{eq:Aclassadj}. 
A central element in the proof of Proposition \ref{Aclasslocal} will be Theorem  \ref{Apropone} proven in Section \ref{Aclassproduct}, which provides an explicit formula for the universal relative  Atiyah class of 
$\CX$ in terms of the universal Atiyah class of $\CM$. 
To begin with, note the following restriction Lemmas. 
\begin{lem}\label{restrlem} 
 \begin{enumerate}[label=(\roman*)]
 \item Let $Q$ be an object of $\mathscr{D}^b(\CM)$. Then the natural morphism $L(1\times i_\CY)^*\Delta_{\CX*}\pi_\CM^*Q \to (1\times i_\CY)^*\Delta_{\CX*}\pi_\CM^*Q$ is an isomorphism in $\mathscr{D}^b(\CX\times \CY)$. Moreover there is a further isomorphism 
\be\label{eq:restrisomA} 
(1\times i_\CY)^*\Delta_{\CX*}\pi_\CM^*Q\cong \Gamma_{\CY*}\rho_\CM^*Q. 
\ee 

\item Let $r_{\CY} : \CX \times \CY \to \CY$ and $p_{i,\CY}: \CY \times \CY \to \CY$, $1\leq i \leq 2$,  denote the canonical projections. There is a natural isomorphism 
\be\label{eq:restrisomB} 
L(i_\CY\times 1)^* r_{\CY}^*\IG \cong p_{2, \CY}^*\IG. 
\ee
\end{enumerate}
\end{lem}
\begin{proof}
(\textit{i}). Consider Diagram \eqref{cartesian}. Recall that $\CX$ is the total space of the line bundle 
$\CK= \rho_Y^*K_Y$ on $\CY$. Let $q_\CY:\CX \to \CY$ denote the 
natural projection, and let $\zeta \in H^0(\CX, q_\CY^*\CK)$ denote 
the tautological section. 
Then $\CX\times \CX$ is the total space of the line bundle $r_{\CY}^*\CK$ on $\CX\times \CY$, such that the canonical projection is $1\times q_\CY: \CX \times \CX \to \CX \times \CY$. 
\begin{equation}\label{cartesian}
\xymatrix{\CX \times \CX \ar[rr]^{1\times q_{\CY}}\ar[d]^{\pi_{1,\CX}}&&\CX\times \CY\ar[d]^{r_{\CY}}\\
\CX\ar[rr]^{q_{\CY}}\ar[dd]^{\pi_{X}}\ar[rd]^{\pi_{\CM}}&& \CY\ar[dd]^{\rho_{\CY}}\ar[ld]_{\rho_{\CM}}\\
&\mathcal{M}&\\
X\ar[rr]^{q}&& Y\\}
\end{equation}
Moreover, there is a canonical isomorphism $(1\times q_\CY)^*r_\CY^*\CK \cong \pi_{1,\CX}^* q_\CY^*\CK$ which identifies the tautological section of the former with $\pi_{1,\CX}^*\zeta$. 
Finally, 
the image 
of the closed embedding $1\times i_\CY: \CX \times \CY \to \CX \times \CX$ coincides scheme theoretically with the zero 
locus of the tautological section $\pi_{1,\CX}^*\zeta$. This implies that the sheaf $(1\times i_\CY)_*\CO_{\CX \times \CY}$ has a two term locally free resolution 
\[
\xymatrix{ 
\pi_{1,\CX}^* q_\CY^*\CK^{-1} \ar[r]^-{\pi_\CX^*\zeta} & \CO_{\CX\times \CX} \\}
\]
Taking the tensor product with $\Delta_{\CX*}\pi_\CM^*Q$, one obtains the two term complex
\[
\xymatrix{ 
\Delta_{\CX*}(q_\CY^*\CK^{-1}\otimes \pi_{\CM}^*Q) 
\ar[rr]^-{\Delta_{\CX*}(\zeta \otimes {\bf 1})} & &
\Delta_{\CX*}(\pi_{\CM}^*Q)\\}
\]
However, the morphism $\zeta \otimes {\bf 1}: q_\CY^*\CK^{-1}\otimes \pi_{\CM}^*Q\to \pi_{\CM}^*Q$ is injective ( as a morphism in the category complexes, $K(\CX)$) since it has an injective restriction to each fiber $X_m$, $m \in \CM$.
This implies that all higher local Tor sheaves, resulting from derived pullback of $\Delta_{\CX*}\pi_\CM^*Q$ to $\CX\times \CY$ vanish, and the natural projection $L(1\times i_\CY)^*\Delta_{\CX*}\pi_\CM^*Q \to (1\times i_\CY)^*\Delta_{\CX*}\pi_\CM^*Q$ is an isomorphism as claimed. 

In order to prove \eqref{eq:restrisomA}, note the canonical isomorphism 
\[ 
\pi_\CM^*Q \cong \Delta_\CX^*\Delta_{\CX*} \pi_\CM^*Q. 
\]
This yields 
\[ 
\rho_\CM^*Q \cong i_\CY^*\pi_\CM^*Q \cong i_\CY^* \Delta_\CX^*\Delta_{\CX*} \pi_\CM^*Q \cong \Gamma_{\CY}^* 
(1\times i_\CY)^*\Delta_{\CX*} \pi_\CM^*Q.
\]
Hence 
\[ 
\Gamma_{\CY*} \rho_\CM^*Q \cong  \Gamma_{\CY*} \Gamma_\CY^* 
(1\times i_\CY)^*\Delta_{\CX*} \pi_\CM^*Q.
\]
The right hand side of the above relation is further isomorphic to 
\[
(1\times i_\CY)^*\Delta_{\CX*} \pi_\CM^*Q/ I_{\Gamma_\CY} \cdot (1\times i_\CY)^*\Delta_{\CX*} \pi_\CM^*Q
\]
where $I_{\Gamma_\CY}\subset \CO_{\CX\times \CY}$ is the ideal sheaf 
associated to the closed embedding $\Gamma_\CY$. However, since 
the right square in \eqref{eq:XYsquares} 
is cartesian, the pull-back $(1\times i_\CY)^*\Delta_{\CX*} \pi_\CM^*Q$ is annihilated by $I_{\Gamma_\CY}$, hence one obtains an isomorphism 
\[
\Gamma_{\CY*} \rho_\CM^*Q \cong (1\times i_\CY)^*\Delta_{\CX*} \pi_\CM^*Q.
\]

(\textit{ii}). As shown in Lemma \ref{flatlemma}, the pull-back $r_{\CY}^*\IG$ is flat over $\CX$. This implies the claim. 

\end{proof}

\begin{lem}\label{pushfwdlem} 
Let 
\[
\pi_{i,\CX} :\CX \times \CX \to \CX, \qquad 
p_{i,\CY} :\CY \times \CY \to \CY, \qquad 1\leq i \leq 2,
\]
and 
\[
r_{\CX}: \CX \times \CY \to \CX,\qquad r_{\CY}:\CX \times \CY \to \CY
\]
denote the canonical projections. Let $Q$ be an object of $\mathscr{D}^b(\CM)$. 
Then there are isomorphisms 
\be\label{eq:pushfwdisomA} 
p_{2,\CX}^*i_{\CY*}\IF \otimes^L 
\Delta_{{\CX}*}(\pi_\CM^*Q) \cong (1\times i_\CY)_*(r_{\CY}^*\IG \otimes^L \Gamma_{\CY*}(\rho_\CM^*Q) )
\ee 
in $\mathscr{D}^b(\CX\times \CX)$, respectively
\be\label{eq:pushfwdisomB}
r_{\CY}^*\IG \otimes^L \Gamma_{\CY*}(\rho_\CM^*Q) \cong 
(i_\CY\times 1)_*(p_{2,\CY}^*\IG\otimes^L \Delta_{\CY*}(\rho_\CM^*Q)).
\ee
in $\mathscr{D}^b(\CX\times \CY)$. 
\end{lem} 

\begin{proof}
Note the commutative diagram
\be\label{eq:FMsquareA} 
\xymatrix{ 
\CX \times \CY \ar[rr]^{1\times i_\CY} \ar@/_1pc/[dd]_-{r_{\CY}} & & \CX \times \CX 
\ar@/_1pc/[dd]_-{p_{2,\CX}} \\ & & \\
\CY \ar[rr]^-{i_\CY} \ar@/_1pc/[uu]_-{\Gamma_\CY} \ar[rd]_{\rho_{\mathcal{M}}}& & \CX \ar@/_1pc/[uu]_-{\Delta_{\CX}}\ar[ld]^{\pi_{\mathcal{M}}}\\
&\mathcal{M}&}
\ee
where both squares are cartesian. Since $\IF = i_{\CY*}\IG$ 
and  $p_{2,\CX}$ is flat, there is an isomorphism 
\[
p_{2,\CX}^*\IF \cong (1\times i_\CY)_*r_{\CY}^*\IG.
\]
Then the projection formula \cite[Prop. II.5.6]{H66} yields an isomorphism 
\be\label{eq:FMisomA} 
\bal 
p_{2,\CX}^*i_{\CY*}\IF \otimes^L 
\Delta_{{\CX}*}(\pi_\CM^*Q) \cong (1\times i_\CY)_*
(r_{\CY}^* \IG \otimes^L 
L(1\times i_\CY)^*\Delta_{{\CX}*}(\pi_\CM^*Q)).
\eal 
\ee
As shown in Lemma \ref{restrlem}.$i$, the derived restriction $L(1\times i_\CY)^*\Delta_{{\CX}*}(\pi_\CM^*Q))$ coincides with the ordinary restriction and
it is furthermore isomorphic to $\Gamma_{\CY*}(\rho_\CM^*Q)$.
Therefore \eqref{eq:FMisomA} yields 
\be\label{eq:FMisomB}   
\bal 
p_{2,\CX}^*\IF \otimes^L 
\Delta_{{\CX}*}(\pi_\CM^*Q) \cong 
(1\times i_\CY)_*
(r_{\CY}^*\IG \otimes^L
\Gamma_{\CY*}(\rho_\CM^*Q)). \\
\eal
\ee
Next note that 
\be\label{eq:FMisomD}
\Gamma_{\CY*}(\rho_\CM^*Q)\cong (i_\CY\times 1)_* \Delta_{\CY*}(\rho_\CM^*Q)
\ee
since $\Gamma_\CY = (i_\CY \times 1)\circ \Delta_{\CY}$. 
Using again the projection formula and Lemma 
\ref{restrlem}.\textit{ii}, this further yields the isomorphism \eqref{eq:pushfwdisomB}. 
\end{proof}
\hfill
\subsection{Proof of Proposition \ref{Aclasslocal}.} 
By Theorem \ref{Apropone}, the relative Atiyah class 
$\alpha_{\CX/X}$ is given by 
\[ 
\alpha_{\CX/X} = \Delta_{\CX*}(\pi_\CM^*\alpha_\CM) : 
\Delta_{\CX*}(\pi_\CM^*\CO_\CM) \to \Delta_{\CX*}(\pi_\CM^*\IL^{\bullet}_\CM[1]).
\]
Using isomorphisms \eqref{eq:pushfwdisomA} and \eqref{eq:pushfwdisomB},
and the functorial properties of the projection formula, 
one obtains the commutative diagrams
\be\label{eq:AclassdiagFA} 
\xymatrix@C=.5em{ 
p_{2,\CX}^*\IF \otimes^L 
 \Delta_{\CX*}(\pi_\CM^*\CO_\CM) \ar[rrrr]^-{{\bf 1}\otimes \alpha_{\CX/X}} \ar[d]^-{\wr} 
 & & & & p_{2,\CX}^*\IF \otimes^L \Delta_{\CX*}(\pi_\CM^*\IL^{\bullet}_\CM[1])\ar[d]^-{\wr} \\
(1\times i_\CY )_*(r_{\CY}^*\IG \otimes^L \Gamma_{\CY*} \rho_\CM^*\CO_\CM) 
 \ar[rrrr]^{\textbf{f}_{1}}
  & & &  &
 (1\times i_\CY)_*(r_{\CY}^*\IG \otimes^L \Gamma_{\CY*}\rho_\CM^*\IL^{\bullet}_\CM[1]) \\
 }
 \ee
 and 
 \be\label{eq:AclassdiagFB} 
\xymatrix@C=.5em{ 
r_{\CY}^*\IG \otimes^L 
 \Gamma_{\CY*}(\rho_\CM^*\CO_\CM) \ar[rrrr]^-{({\bf 1}\otimes \Gamma_{\CY})_{*}(\rho_\CM^*\alpha_\CM)} \ar[d]^-{\wr} 
 & & & & r_{\CY}^*\IG\otimes^L \Gamma_{\CY*}(\rho_\CM^*\IL^{\bullet}_\CM[1])\ar[d]^-{\wr} \\
(i_\CY\times 1)_*(p_{1,\CY}^*\IG \otimes^L \Delta_{\CY*}\rho_\CM^*\CO_\CM) 
 \ar[rrrr]^{\textbf{f}_{2}}
  & & &  &
 (i_\CY\times 1)_*(p_{2,\CY}^*\IG \otimes^L \Delta_{\CY*}\rho_\CM^*\IL^{\bullet}_\CM[1]) \\
 }
 \ee 
where in $\textbf{f}_{1}, \textbf{f}_{2}$ in the above diagrams 
denote the morphisms 
\[
 \textbf{f}_{1}=(1\times i_\CY)_*(({\bf 1}\otimes \Gamma_{\CY})_{*}(\rho_\CM^*\alpha_\CM)), \qquad 
 \textbf{f}_{2}=(i_\CY\times 1)_*(({\bf 1}\otimes \Delta_{\CY})_{*}(\rho_\CM^*\alpha_\CM)).
 \]
 Furthermore, the left and right vertical arrows in the above diagrams are specializations 
of the isomorphisms \eqref{eq:pushfwdisomA} and \eqref{eq:pushfwdisomB}
to $Q=\CO_\CM$ and $Q=\IL^{\bullet}_\CM[1]$ respectively.  
Note also that $ \Delta_{\CY*}(\rho_\CM^*\alpha_\CM)=\alpha_{\CY/Y}$ according to Theorem  \ref{Apropone}. 

Now consider the commutative diagrams 
\be\label{eq:FMsquareB} 
\xymatrix{ 
\CY \times \CY \ar[rr]^{i_\CY \times 1} \ar[dd]_-{p_{1,\CY}} & &  \CX \times \CY 
\ar[dd]_-{r_{\CX}} \\ & & \\
\CY \ar[rr]^-{i_\CY} & & \CX \\}\qquad\qquad 
\xymatrix{ 
\CX \times \CY \ar[rr]^{1\times i_\CY } \ar[dd]_-{r_{\CX}} & &  \CX \times \CX 
\ar[dd]_-{p_{1,\CX}} \\ & & \\
\CX \ar[rr]^-{1} &  & \CX \\}
\ee
Using again isomorphisms \eqref{eq:pushfwdisomA} and \eqref{eq:pushfwdisomB}, one obtains successively 
\be\label{eq:FMisomC} 
\bal 
\IF\otimes^L \pi_{\CM}^*Q & \ \cong 
\dR{p_{1,\CX*}}(p_{2,\CX}^*\IF \otimes^L \Delta_{{\CX}*}(\pi_\CM^*Q)) \\
&\  \cong \dR{p_{1,\CX*}}
(r_{\CY}^*\IG \otimes^L \Gamma_{\CY*}\rho_\CM^*Q) \\
&\  \cong \dR{p_{1,\CX*}} (i_\CY\times 1)_*(p_{2,\CY}^*\IG \otimes^L \Delta_{\CY*}(\rho_\CM^*Q))
\\
& \ \cong i_{\CY*} \dR_{p_{1,\CY*}}(p_{2,\CY}^*\IG \otimes^L \Delta_{\CY*}\rho_\CM^*Q) \\
& \ \cong i_{\CY^*} (\IG \otimes^L \rho_{\CM}^* Q)
\eal
\ee
for any object $Q$ of $\mathscr{D}^b(\CM)$. Then using diagrams \eqref{eq:AclassdiagFA} and \eqref{eq:AclassdiagFB}
one obtains the commutative diagram 
\be\label{eq:AclassdiagG} 
\xymatrix{ 
\IF \ar[d]^-{\bf 1}  \ar[rrr]^-{\alpha_{\CX/X}(\IF)} & & & \IF \otimes^L \pi_{\CM}^*\IL^{\bullet}_\CM[1] \ar[d]^-{\wr} \\
i_{\CY*}\IG \ar[rrr]^-{i_{\CY*}(\alpha_{\CY/Y}(\IG))} & & &
i_{\CY*}(\IG \otimes^L \rho_{\CM}^*\IL^{\bullet}_\CM[1]) }
\ee
where the left vertical arrow is the canonical isomorphism given by the projection formula. This implies equation 
\eqref{eq:Aclassadj}.

\hfill $\text{q.e.d.}$

\subsection{Comparison of obstruction theories}\label{compobs}

The next step is to translate Proposition \ref{Aclasslocal} into a 
statement on obstruction theories using  \cite[Theorem 4.1]{HT10}. 

Using the canonical isomorphism 
\be\label{eq:dualisomB1}
\mathbf{R}{\rm Hom}_{\mathscr{D}^b(\CX)}(\IF, \IF \otimes^L \pi_\CM^*\IL^{\bullet}_\CM)[1] 
\cong \mathbf{R}{\rm {H}om}_{\mathscr{D}^b(\CX)}(\mathbf{R}\mathscr{H}om_{\mathscr{D}^b(\CX)}(\IF, \IF), 
\pi_\CM^*\IL^{\bullet}_\CM)[1].
\ee
the relative Atiyah class $\alpha_{\CX/X}(\IF)$ 
is identified with a map 
\be\label{eq:betamapA}
\beta_X: \mathbf{R}\mathscr{H}om_{\mathscr{D}^b(\CX)}(\IF, \IF) \to \pi_\CM^*\IL^{\bullet}_\CM[1].
\ee
In complete analogy, the relative Atiyah class $\alpha_{\CY/Y}$ 
is also identified with a map 
\be\label{eq:betamapB}
\beta_Y: \mathbf{R}\mathscr{H}om_{\mathscr{D}^b(\CY)}(\IG, \IG) \to \rho_\CM^*\IL^{\bullet}_\CM[1].
\ee

Since $X$ is $K$-trivial, as in \cite[Sect 4.2]{HT10},  
Grothendieck-Verdier duality yields an isomorphism
\begin{align}
&
\dR{\rm {H}om_{\mathscr{D}^b(\CX)}}(\mathbf{R}\mathscr{H}om_{\mathscr{D}^b(\CX)}(\IF, \IF), 
\pi_\CM^*\IL^{\bullet}_\CM[1])\notag\\
&
 \cong \mathbf{R}{\rm {H}om}_{\mathscr{D}^b(\CM)}(\dR{\mathscr Hom}_{\pi_{\CM}}(\IF, \IF), \IL^{\bullet}_\CM)[1-d].
\end{align}
where $d={\rm dim}(X)$. Therefore the relative Atiyah class 
is further identified with a map 
\[
a_X: \mathbf{R}{\mathscr{H}om}_{\pi_\CM}(\IF, \IF)[d-1]\to \IL^{\bullet}_\CM.
\]
As shown in \cite[Thm. 4.1]{HT10}, this map is an obstruction 
theory for $\CM$, as defined in \cite[Definition 4.4]{BF97}.

Again, in complete analogy, the relative Atiyah class 
$\alpha_{\CY/Y}(\IG)$ yields a second obstruction theory 
\[
a_Y: \mathbf{R}\mathscr{H}om_{\rho_\CM}(\IG, \IG\otimes \CK))[d-2]\to \IL^{\bullet}_\CM,
\]
where $\rho_\CM:\CY\to \CM$ is the canonical projection, and 
$\CK$ is the pull-back
$\CK = \rho_{Y}^*K_Y$. Moreover, as shown in Lemma \ref{adjlemB} there is a natural  isomorphism 
\be\label{eq:RHomsplitA}
\bal
& \mathbf{R}\mathscr{H}om_{\mathscr{D}^b(\CX)}(\IF, \IF) \cong 
\\
& i_{\CY*} \mathbf{R}\mathscr{H}om_{\mathscr{D}^b(\CY)}(\IG\otimes \CK^{-1},\IG)[-1]\oplus i_{\CY*} \mathbf{R}\mathscr{H}om_{\mathscr{D}^b(\CY)}(\IG,\IG).\\
\eal 
\ee
This yields a decomposition 
\be\label{eq:RHomsplitAB}
\bal
& \mathbf{R}\mathscr{H}om_{\pi_\CM}(\IF, \IF)[d-1] \cong \\
& \mathbf{R}\mathscr{H}om_{\rho_\CM}(\IG,\IG\otimes \CK))[d-2] \oplus 
\mathbf{R}\mathscr{H}om_{\rho_\CM}(\IG,\IG)[d-1].\\
 \eal
\ee
The goal of this section is to prove: 

\begin{prop}\label{obsprop} 
The map $a_X$ has components $a_X = (a_Y,\ 0)$ with respect to the 
decomposition \eqref{eq:RHomsplitAB}.
\end{prop}

{\it Proof}. 
Proposition \ref{obsprop} follows from Proposition 
\ref{Aclasslocal} and identities similar to \eqref{eq:dualisomB1}. 
In addition to \eqref{eq:RHomsplitAB}, 
Lemma \ref{adjlem} also yields an isomorphism 
\be\label{eq:RHomsplitAC}
\bal 
\mathbf{R}\mathscr{H}om_{\mathscr{D}^b(\CX)}(\IF, & \IF\otimes^L \pi_\CM^*\IL^{\bullet}_\CM[1])
 \cong\\
 & i_{\CY*} \mathbf{R}\mathscr{H}om_{\mathscr{D}^b(\CY)}(\IG,\IG\otimes^L \rho_\CM^*\IL^{\bullet}_\CM[1])\\
 &  \oplus i_{\CY*} \mathbf{R}\mathscr{H}om_{\mathscr{D}^b(\CY)}(\IG\otimes \CK^{-1},\IG\otimes^L \rho_\CM^*\IL^{\bullet}_\CM)).\\
 \eal
\ee
This identity is canonically equivalent to 
\be\label{eq:RHomsplitAD} 
\bal 
\mathbf{R}{\mathscr{H}om}_{\mathscr{D}^b(\CX)}( & \mathbf{R}\mathscr{H}om_{\mathscr{D}^b(\CX)}(\IF, \IF), \ 
\pi_\CM^*\IL^{\bullet}_\CM[1]) \cong \\ 
& i_{\CY*} \mathbf{R}\mathscr{H}om_{\mathscr{D}^b(\CY)}(\mathbf{R}\mathscr{H}om_{\mathscr{D}^b(\CY)}(\IG,\IG), \rho_\CM^*\IL^{\bullet}_\CM[1]) \ \oplus \\ & 
i_{\CY*} \mathbf{R}\mathscr{H}om_{\mathscr{D}^b(\CY)}(\mathbf{R}\mathscr{H}om_{\mathscr{D}^b(\CY)}(\IG,\IG\otimes \CK), \rho_\CM^*\IL^{\bullet}_\CM)\\
\eal 
\ee

Next, applying the global derived functor 
$\dR\Gamma_{{\mathscr D}^b(\CX)}$ to both sides of 
equation 
\eqref{eq:RHomsplitAD} yields the isomorphism 
\[
\bal 
\mathbf{R}{\rm {H}om}_{\mathscr{D}^b(\CX)}( & \mathbf{R}\mathscr{H}om_{\mathscr{D}^b(\CX)}(\IF, \IF), \ 
\pi_\CM^*\IL^{\bullet}_\CM[1]) \cong \\ 
&  \mathbf{R}{\rm {H}om}_{\mathscr{D}^b(\CY)}(\mathbf{R}\mathscr{H}om_{\mathscr{D}^b(\CY)}(\IG,\IG), \rho_\CM^*\IL^{\bullet}_\CM[1]) \ \oplus \\ & 
\mathbf{R}{\rm {H}om}_{\mathscr{D}^b(\CY)}(\mathbf{R}\mathscr{H}om_{\mathscr{D}^b(\CY)}(\IG,\IG\otimes \CK), \rho_\CM^*\IL^{\bullet}_\CM)\\
\eal 
\]
Then Proposition \eqref{Aclasslocal} implies that 
$$\beta_{\CX/X}(\IF): \mathbf{R}\mathscr{H}om_{\mathscr{D}^b(\CX)}(\IF,\IF)\to \pi_\CM^*\IL^{\bullet}_\CM[1]$$ has components 
\be\label{eq:betacomp}
\beta_{X} = (\beta_Y, \ 0) 
\ee
with respect to the above decomposition, where 
$\beta_X, \beta_Y$ are the maps obtained in equations 
\eqref{eq:betamapA}, \eqref{eq:betamapB} respectively.

Now note that $Li_\CY^*\pi_\CM^*\IL^{\bullet}_\CM \cong \rho_\CM^*\IL^{\bullet}_\CM$, as shown in equation \eqref{eq:dpbid}. 
This yields the isomorphism 
\[ 
i_\CY^! \pi_\CM^*\IL^{\bullet}_\CM \cong \CK \otimes \rho_\CM^*\IL^{\bullet}_\CM [-1].
\]
Then using Grothendieck-Verdier duality for 
the closed embedding $i_\CY:\CY \to \CX$, identity \eqref{eq:RHomsplitAD} 
is further equivalent to
\be\label{eq:RHomsplitB}
\bal 
\mathbf{R}\mathscr{H}om_{\mathscr{D}^b(\CX)}(\mathbf{R}\mathscr{H}om_{\mathscr{D}^b(\CX)}&(\IF, \IF),\  
\pi_\CM^*\IL^{\bullet}_\CM[1]) \cong\\  
& \mathbf{R}\mathscr{H}om_{\mathscr{D}^b(\CX)}
(i_{\CY*}(\mathbf{R}\mathscr{H}om_{\mathscr{D}^b(\CY)}(\IG,\IG\otimes \CK), \pi_\CM^*\IL^{\bullet}_\CM[2])\ \oplus \\
& \mathbf{R}\mathscr{H}om_{\mathscr{D}^b(\CX)}(i_{\CY*}\mathbf{R}\mathscr{H}om_{\mathscr{D}^b(\CY)}(\IG,\IG), \pi_\CM^*\IL^{\bullet}_\CM[1]).\\
\eal 
\ee
Finally, using Grothendieck-Verdier duality for the projection $\pi_\CM:\CX\to \CM$ in \eqref{eq:RHomsplitB}, one obtains the 
decomposition 
\be\label{eq:RHomsplitD}
\bal 
\dR{\rm {H}om}_{\mathscr{D}^b(\CM)}(& \mathbf{R}{\mathscr{H}om}_{\pi_\CM}(\IF, \IF), 
\  \IL^{\bullet}_\CM)[1-d] \cong\\
& \dR{\rm{H}om}_{\mathscr{D}^b(\CM)}( \mathbf{R}\mathscr{H}om_{\rho_\CM}(\IG, \IG\otimes \CK),\ \IL^{\bullet}_\CM)[2-d] \ \oplus \\
&  \dR{\rm{H}om}_{\mathscr{D}^b(\CM)}(\mathbf{R}\mathscr{H}om_{\rho_\CM}(\IG, \IG), \ \IL^{\bullet}_\CM)[1-d].\\
\eal
\ee
This is in fact the global version of \eqref{eq:RHomsplitAB}. 
Then  equation \eqref{eq:betacomp} implies that $a_X$ has indeed components $(a_Y,0)$ with respect to  decomposition \eqref{eq:RHomsplitAB}. 

\hfill $\Box$

\section{Donaldson-Thomas invariants for sheaves on local fourfolds}\label{DTsect} 

In this paper, the main application of Theorem \ref{obsthm} is
the construction of virtual counting invariants of sheaves on 
a local Calabi-Yau fourfolds, which is presented below. 

 Let $S$ be a smooth projective surface, let $\CO_S(1)$ be a very ample line bundle on $S$ and let $h = c_1(\CO_S(1))$. Throughout this paper it will be further assumed that 
$H^{1}(\CO_S)=0$ and that the integral cohomology of $S$ is torsion free. 
In the framework of Section \ref{obstrsect} let $Y$ be the total space of $K_S(D)$, where $D\in {\rm Pic}(S)$ be an effective non-zero divisor on $S$. Then $K_Y\cong q^*\CO_S(-D)$, where $\pi_{S}:Y \to S$ 
denotes the canonical projection. Hence $X$ is isomorphic to the 
total space of the rank two bundle $V=K_S(D) \oplus \CO_S(-D)$ on $S$. 
Let $g = \pi_{S}\circ q : X \to S$ denote the canonical projection. 

\subsection{Stable two dimensional sheaves}\label{stabsheaves} 
Let $\Coh_c(X)$ denote the abelian category of coherent sheaves on $X$ with proper support. Since the projection $g: X\to S$ is affine, for any such sheaf $\mathscr{F}$ the derived pushforward $\dR_{g*}\mathscr{F}$ is isomorphic to the ordinary pushforward, $g_*F$. Therefore using the Grothendieck-Riemann-Roch theorem, the topological invariants 
of $\mathscr{F}$ are completely determined by the total Chern class $c(g_*\mathscr{F})\in H^{\rm even}(S,\IZ)$. Using Poincar{\'e} duality on $S$, 
the latter is canonically identified with an element of 
$H_{\rm even}(S,\IZ)$, which can be written as 
\[
(r(\mathscr{F})[S], \beta(F), n(F)[pt]) \in H_4(S,\IZ) \oplus H_2(S,\IZ) \oplus H_0(S, \IZ). 
\]
Since the generators $[S], [pt]$ are canonical, this is further identified with the triple 
\[
\gamma(F) = (r(\mathscr{F}), \beta(F), n(F)) \in \IZ \oplus H_2(S,\IZ) \oplus \IZ.
\]
In particular  $r(\mathscr{F}) = {\rm rk}(g_*\mathscr{F})$ will 
be informally referred to as the rank of $\mathscr{F}$ 
over $S$ in the following.

Next note that the polarization $h$ 
determines naturally a stability condition on $\Coh_c(X)$ using the Hilbert polynomial 
\[ 
P_h(\mathscr{F}; m) = \chi(\mathscr{F}\otimes g^*\CO_S(m)). 
\]
As in \cite[Definition 1.2.4]{HL97}, $\mathscr{F}$ is $h$-(semi)stable if and only if $\mathscr{F}$ is pure, and 
\[ 
p_h(\mathscr{F}';m)\ (\leq)\  p_h (\mathscr{F};m) \qquad  m >>0
\]
for any proper nonzero subsheaf $\mathscr{F}'\subset \mathscr{F}$. Here $p_h(\mathscr{F};m)$ 
denotes the reduced Hilbert polynomial of $\mathscr{F}$. 

In terms of 
$\gamma(\mathscr{F}) = (r(\mathscr{F}), \beta(\mathscr{F}), n(\mathscr{F}))$ the Hilbert polynomial reads 
\[ 
\bal 
& P_h(\mathscr{F}; m) = \\
& {r(\mathscr{F})h^2\over 2}\, m^2 + h\cdot\left(\beta(\mathscr{F})+{r(\mathscr{F}) c_1(S)\over 2}\right))\, m + {\beta\cdot (\beta+ c_1(S))\over 2} - n(\mathscr{F}) + r(\mathscr{F})\chi(\CO_S).\\
\eal 
\]
Hence, for sheaves with 
$r(\mathscr{F})>0$, 
\[
\bal 
& p_h(\mathscr{F},m) = \\
& m^2 + \left({2 h \cdot \beta(\mathscr{F}) \over r(\mathscr{F})h^2}+ 
{h\cdot c_1(S)\over h^2}\right)\, m + {2\over r(\mathscr{F})h^2} 
\left({\beta\cdot (\beta+ c_1(S))\over 2} - n(\mathscr{F})\right) 
+ {2\chi(\CO_S)\over h^2}.\\
\eal 
\]
For any such sheaf $\mathscr{F}$ let 
\be\label{eq:hslope} 
\mu_h(\mathscr{F}) = {h\cdot \beta(\mathscr{F})\over r(\mathscr{F})}, \qquad \nu_h(\CF)= {1\over r(\mathscr{F})} 
\left({\beta\cdot (\beta+ c_1(S))\over 2} - n(F)\right).
\ee
Then $\mathscr{F}$ is $h$-(semi)stable if and only if it is pure, and any nonzero proper subsheaf $\mathscr{F}'\subset \mathscr{F}$
satisfies 
\be\label{eq:muineq}
\mu_h(\mathscr{F}')\ (\leq)\ \mu_h(\mathscr{F})
\ee
while in case of equality,
\be\label{eq:nuineq}
\nu_h(\mathscr{F}') \ (\leq)\ \nu_h(\mathscr{F}).
\ee

Furthermore, as in \cite[Definition 2.1.12]{HL97}, a pure sheaf 
$\mathscr{F}$ with $r(\mathscr{F})>0$ is said to be $\mu_h$-(semi)stable if and only if any nonzero proper subsheaf $\mathscr{F}'\subset \mathscr{F}$ with $0 < r(\mathscr{F}')<
r(\mathscr{F})$ satisfies inequality \eqref{eq:muineq}. 
As usual, $h$-semistability implies $\mu_h$-semistability for any such sheaf.

\begin{lem}\label{twiststab} 
Let $\mathscr{F}$ be a pure compactly supported sheaf on $X$ with $r(\mathscr{F})>0$. 
Then $\mathscr{F}$ is $\mu_h$-(semi)stable if and only if $\mathscr{F}\otimes g^*\CO_S(-D)$ 
is $\mu_h$-(semi)stable. 
\end{lem} 

{\it Proof.} Note that for any compactly supported sheaf $\mathscr{E}$ on 
$X$ with $r(E)>0$ one has 
\be\label{eq:twistslope}
\bal 
\mu_h(\mathscr{E}\otimes g^*\CO_S(-D)) & = \mu_h(\mathscr{E}) - h\cdot D,
\eal
\ee
Then the claim follows by a straightforward verification of the $\mu_h$-(semi)stability condition. 

\hfill $\Box$ 

\begin{cor}\label{vanishingcor} 
Let $\mathscr{F}$ be a pure $\mu_h$-semistable compactly supported sheaf on $X$ with $r(\mathscr{F})>0$. 
Then the following vanishing result holds
\be\label{eq:vanishingA}
{\rm Ext}^0_X(\mathscr{F}, \mathscr{F} \otimes g^*\CO_S(-D)) =0.
\ee
In particular this holds for all $h$-semistable sheaves 
$\mathscr{F}$ with $r(\mathscr{F})>0$. 
\end{cor} 

{\it Proof}. Since $\mathscr{F}$ is $\mu_h$-semistable, as shown in Lemma \ref{twiststab} above, $\mathscr{F}\otimes g^*\CO_S(-D)$ is also $\mu_h$-semistable. Moreover, since $D$ is effective and non-zero, 
 \[
 \mu_h(\mathscr{F}\otimes g^*\CO_S(-D)) = \mu_h(\mathscr{F}) - h \cdot D < \mu_h(\mathscr{F}). 
 \]
Then the claim follows 
from Theorem 1.6.6 and Proposition 1.2.7 in \cite{HL97}.  

\hfill $\Box$

\begin{cor}\label{supplemmaA} 
Let $\mathscr{F}$ be a pure compactly supported $\mu_h$-semistable 
 sheaf on $X$ such that $r(\mathscr{F})>0$. Then $\mathscr{F}$ is scheme theoretically supported on $Y$, and as a sheaf on $Y$, the following holds: 
 \be\label{eq:vanishingAB}
 {\rm Ext}^3_Y(\CF,\CF)=0.
 \ee
 In particular this holds for all $h$-semistable sheaves 
 $\mathscr{F}$ with $r(\mathscr{F})>0$. 
 \end{cor} 
 
 {\it Proof}. 
 Note that $Y\subset X$ is the zero locus of the tautological section $$\xi \in H^0(X, g^*\CO_S(-D)).$$  
 In order to prove the first claim it suffices to show that the morphism 
 $${\bf 1}_\mathscr{F}\otimes\xi: \mathscr{F} \to \mathscr{F} \otimes g^*\CO_S(-D)$$ 
 is identically zero for any sheaf $\mathscr{F}$ as in Lemma \ref{supplemmaA}. This follows from the first vanishing 
 result in \eqref{eq:vanishingA}. 
 
  The vanishing result \eqref{eq:vanishingAB} follows by Serre duality for compactly supported sheaves, noting that $K_Y \cong \pi_S^*\CO_S(-D)\simeq g^*\CO_S(-D)|_Y$. 
 
\hfill $\Box$

\subsection{Reduced obstruction theory}\label{redobs} 

As explained in Section \ref{compobs}, the moduli space $\CM$ has two natural obstruction theories, associated to deformations of sheaves on $X$ and $Y\subset X$ respectively.
Using the Atiyah class formalism of 
\cite[Thm 4.1]{HT10}, these are given by the maps 
 \be\label{eq:Xobs} 
a_X: \mathbf{R}\mathscr{H}om_{\pi_\CM}(\IF, \IF)[3] \to \IL^{\bullet}_\CM
\ee
and 
\be\label{eq:Yobs}
a_Y:  \mathbf{R}\mathscr{H}om_{\rho_\CM}(\IG, \IG\otimes \rho_Y^*K_Y))[2]\to \IL^{\bullet}_\CM
\ee
respectively. As in Section \ref{obstrsect}, here $\pi_\CM, \rho_\CM$ denote the canonical projections $X \times \CM \to \CM$ and $Y \times \CM \to \CM$ respectively. The projections onto the $X$ and $Y$ factors will be respectively denoted by $p_X$ and $\rho_Y$. 
Proposition \ref{obsprop} shows that the first obstruction theory splits as a direct sum $a_X = (a_Y, 0)$ 
with respect to the decomposition \eqref{eq:RHomsplitAB}, 
which in the present case reads
\be\label{eq:RHomsplitAG}
\bal 
& \mathbf{R}\mathscr{H}om_{\pi_\CM}(\IF, \IF)[3]\cong \\
& 
\mathbf{R}\mathscr{H}om_{\rho_\CM}(\IG, \IG\otimes \rho_Y^*K_Y)[2]
\oplus \mathbf{R}\mathscr{H}om_{\rho_\CM}(\IG, \IG)[3].\\
\eal
\ee
In particular the obstruction theory \eqref{eq:Xobs} of $\CM$ as a moduli space of sheaves on $X$ can be naturally truncated to 
the \eqref{eq:Yobs} without 
any loss of information.
This result is similar to \cite[Lemma 6.4]{CL14}, except that here it is proven for global obstruction theories
as opposed to local 
complex analytic Kuranishi maps. 

The next observation is that the obstruction theory \eqref{eq:Yobs} can be further 
truncated as shown in \cite[Section 4.4]{HT10}. Namely, Grothendieck duality for the projection 
$\rho_\CM: \CY \to \CM$ yields an isomorphism 
\be\label{eq:dualisomB}
 \mathbf{R}\mathscr{H}om_{\rho_\CM}(\IG, \IG\otimes \rho_Y^*K_Y))[2] 
\cong \mathbf{R}\mathscr{H}om_{\rho_\CM}(\IG, \IG)^\vee[-1]. 
\ee
There is also a natural map of complexes \[ 
\CO_\CM \to  \mathbf{R}\mathscr{H}om_{\rho_\CM}(\IG, \IG)
\]
which restricts to the identity map on each fiber $Y_m$, $m\in \CM$. The cone of this map defines the truncation $\tau^{\geq 1}\mathbf{R}\mathscr{H}om_{\rho_\CM}(\IG, \IG)$. Moreover, using the vanishing result \eqref{eq:vanishingAB}, 
the global version of 
Nakayama's lemma  \cite[Lemma 4.2]{HT10} yields a further truncation $\tau^{[1,2]} \mathbf{R}\mathscr{H}om_{\rho_\CM}(\IG, \IG)$ of amplitude $[1, \ 2]$. As shown in \cite[Sect. 4.4]{HT10} the map \eqref{eq:Yobs} also admits a truncation 
\be\label{eq:YobsB} 
\tau^{[-1,0]} a_Y : (\tau^{[1,2]} \mathbf{R}\mathscr{H}om_{\rho_\CM}(\IG, \IG))^\vee[-1] \to \IL^{\bullet}_\CM 
\ee
which is still an obstruction theory for $\CM$. Moreover, this obstruction theory is perfect i.e. it admits a two term locally free resolution. However, as observed \cite{GSY17b}, \cite{TT17a}, the resulting virtual cycle invariants would be identically zero 
since the obstruction sheaf induced by the above obstruction theory contains a trivial direct summand. This problem is solved in loc. cit. by constructing 
a reduced perfect obstruction theory.

More precisely, it is shown in \cite[Proposition 2.4]{GSY17b}, 
and \cite[Theorem 6.1, Theorem  6.5]{TT17a} that there
 is a further splitting 
\be\label{eq:splitobsB}
\tau^{[-1,0]}a_Y = (a_Y^{\rm red}, 0): \begin{array}{c}
(\tau^{[1,2]}  \mathbf{R}\mathscr{H}om_{\rho_\CM}(\IG, \IG))^\vee[-1]_\perp \\ \oplus \\ 
\rho_{\CM*}(\pi_S^*K_S)[2] \\ \end{array} \to \IL^{\bullet}_\CM,
\ee
where $\pi_S: S \times \CM \to \CM$ is the natural projection. 
Furthermore, it is also proven in loc. cit. that the map 
\be\label{eq:Yobsred}
 a_Y^{\rm red} :  (\tau^{[1,2]}  \mathbf{R}\mathscr{H}om_{\rho_\CM}(\IG, \IG))^\vee[-1]_\perp \to \IL^{\bullet}_\CM
 \ee
 is still a perfect obstruction theory for $\CM$, hence it yields a reduced virtual cycle $[\CM]^{\rm vir}_{\rm red}$. 
 
As stated in Corollary \ref{obscor}, it is now clear that the fourfold obstruction theory \eqref{eq:Xobs} admits a reduction 
to \eqref{eq:Yobsred}, which is moreover perfect. 

\subsection{Reduced Donaldson-Thomas invariants}\label{redDT}
The Donaldson-Thomas invariants for sheaves on $X$ can be naturally defined by integration against the reduced 
virtual cycle associated to \eqref{eq:Yobsred}, 
provided that compactness issues are properly addressed. Again, this is similar to \cite[Theorem 6.5]{CL14}, although the logic employed here is slightly different.
 
In the present situation, compactness issues are easily addressed using the natural torus action ${\bf T} \times Y\to Y$ which scales the fibers of the line bundle 
$K_S(D)$ with weight $+1$. Clearly, this admits a lift to a torus action on $X$ 
scaling the fibers of $K_S(D)\oplus \CO_S(-D)$ with weights 
$(+1, -1)$. The fixed point set coincides with the 
zero section, $S \subset Y$, which is obviously proper. This implies that the induced action ${\bf T}\times \CM \to \CM$ on the moduli space has proper fixed locus $\CM^{\bf T}$. Furthermore, by the virtual localization theorem \cite{GP99}, the fixed 
locus has an induced reduced virtual cycle, $[\CM^{\bf T}]^{\rm vir}_{\rm red}$, and an equivariant virtual normal bundle 
\be\label{eq:virnormA} 
\mathscr{N}^{\rm vir} =\left( \tau^{[1,2]} \mathbf{R}\mathscr{H}om_{\rho_\CM}(\IG,\IG)|_{\CM^{\bf T}}\right)^m. 
\ee
Here the superscript $m$ denotes the moving part with respect to the torus action. Then the equivariant residual 
Donaldson-Thomas invariants of $X$ are defined by 
\be\label{eq:DTa} 
DT_X(\gamma) = \int_{[\CM^{\bf T}]^{\rm vir}_{\rm red}} {1\over e_{\bf T} 
(\mathscr{N}^{\rm vir})}.  
\ee

\section{Rank two invariants and universality}\label{rktwosect}

The goal of this section is to prove certain universality properties  of the  equivariant residual invariants \eqref{eq:DTa} 
for rank $r=2$. The surface $S$ will be subject to the same conditions as listed in the beginning of Section 
\ref{DTsect}.

\subsection{Rank two fixed loci}\label{fixedloci}  
Let $S$ be a smooth projective surface as in Section \ref{DTsect}
and $h= c_1(\CO_S(1))$ a polarization on $S$. Note that under the current assumptions the natural map ${\rm Pic}(S) \to H_2(S,\IZ)$ 
is an isomorphism, hence the lattices will be implicitly identified in the following. Moreover, since $H^1(\CO_S)=0$ and 
element $\beta \in H_2(S, \IZ)$ determines a unique line bundle on $S$  up to isomorphism. 

Now let $\gamma = (2, \beta, n)$ with $\beta \in H_2(S,\IZ)$ and 
$n \in \IZ$ and let $\CM = \CM_h(Y, \gamma)$. Let 
$\CM_h(S,\gamma)$ be the moduli space of $h$-stable torsion free 
sheaves on $S$ with topological invariants $\gamma$. For simplicity, in this section set $L=K_S(D)$.  

The rank two torus 
fixed loci in the moduli space $\CM_h(Y,\gamma)$ 
were determined in \cite[Section 3]{GSY17b} and \cite[Section 7]{TT17a}. This section will review the main points using the approach of \cite[Section 3]{GSY17b}. 
As shown in loc. cit.,  for arbitrary rank $r\geq 1$ the
torus fixed loci in the moduli space $\CM$ are classified by 
partitions $\lambda$ of $r$.  For a given partition $\lambda = 
( \lambda_1 \leq \lambda_2 \leq \cdots \leq \lambda_{\ell(\lambda)})$ the corresponding fixed locus consists 
of sheaves $\mathscr{F}$ such that the torsion free sheaf $\mathscr{E}=q_*\mathscr{F}$ has a character decomposition 
\[
\mathscr{E} \cong \bigoplus_{i=0}^{\ell(\lambda)-1} \mathscr{E}_{-i}\otimes {\bf t}^{-i},
\]
where $\mathscr{E}_{-i}$ is a rank $\lambda_{i+1}$ torsion free sheaf on $S$ 
for each $0 \leq i\leq \ell(\lambda)$. Here ${\bf t}$ denotes the 
trivial equivariant line bundle on $S$ equipped with the weight one action of ${\bf T}=\IC^\times$ on each fiber. 
 Moreover, since
\[ 
q_*\CO_Y \cong \bigoplus_{i=0}^{\infty} L^{-i} \otimes {\bf t}^{-i} 
\]
the $\CO_Y$-module structure of $\mathscr{F}$ is specified by a collection 
of morphisms of $\CO_S$-modules 
\[
\psi_i: \mathscr{E}_{-i} \to \mathscr{E}_{-i-1}\otimes L. 
\]
In particular, for $r=2$ one obtains two types of components corresponding to $\lambda = (2)$ and $\lambda=(1,1)$ which will 
be called \textit{type I} and \textit{type II} respectively. The union of type I fixed loci will be denoted by $\CM^{\bf T}_{(2)}$ while the union  of type II fixed loci will be denoted by  $\CM^{\bf T}_{(1,1)}$. 
As shown in \cite[Proposition 3.2]{GSY17b} there is a natural isomorphism of schemes 
\be\label{eq:typeI} 
\CM^{\bf T}_{(2)}\cong \CM_h(S,\gamma).
\ee

Moreover, as shown in \cite[Proposition 3.8]{GSY17b}, the type II fixed locus is isomorphic to a union of nested Hilbert schemes on $S$. Using the notation of loc. cit., given a pair of curve classes $(\alpha_1,\alpha_2) \in H^2(S,\IZ)^{\oplus 2}$ and a pair of integers $(n_1, n_2)\in \IZ^{2}$, $n_1, n_2\geq 0$, let $S_{(\alpha_1, \alpha_2)}^{[n_1,n_2]}$ be the nested Hilbert scheme parametrizing flags of twisted ideal sheaves 
\[ 
\mathcal{I}_1 \otimes L_1 \subset \mathcal{I}_2 \otimes L_2 
\]
where $\mathcal{I}_{1}, \mathcal{I}_{2}$ are ideal sheaves of zero dimensional subschemes 
of $S$ of length $n_1, n_2$ respectively, and $L_1, L_2$ are line bundles  
on $S$ with $c_1(L_i)=-\alpha_i$, $1\leq i \leq 2$. Note in particular that the curve class $\alpha_2-\alpha_1$ must be effective or zero, which will be indicated by  $\alpha_2-\alpha_1 \succeq 0$.  
Then the type II fixed locus is isomorphic to 
\be\label{eq:typeII} 
\CM^{\bf T}_{(1,1)}\cong \bigcup^\prime_{\substack {\beta_1+\beta_2 =\beta\\ \beta_2+c_1(L)-\beta_1\succeq 0\\ }} \ \
\bigcup^\prime_{\substack{n_1,n_2\geq 0\\ 
n_1+n_2 + \beta_1\cdot \beta_2 = n}} \ S_{(\beta_1, \beta_2+c_1(L))}^{[n_1,n_2]} 
\ee
where the ${}^\prime$ superscript indicates that the union in the right hand side is subject to the following additional constraints: 
\begin{itemize} 
\item[$(II.a)$] If $\beta_2+c_1(L)-\beta_1=0$  then $n_1\geq n_2$, and 
\item[$(II.b)$] $h\cdot \beta_2 < h \cdot \beta_1$. 
\end{itemize} 

Using the virtual localization theorem, the residual Donaldson-Thomas invariant $DT_h(X, \gamma)$ then splits as a sum 
\[ 
DT_h(X, \gamma) = DT_h(X, \gamma)_I + DT_h(X, \gamma)_{II}.
\]
The type I contributions are determined in \cite[Proposition 3.2 and Corollary 3.3]{GSY17b} while the type II contributions 
are determined in \cite[Proposition 3.11 and Corollary 3.13]{GSY17b}. In order to write down the resulting formulas, 
it will be convenient to adopt the following notation conventions 
for any component $\CT$ of the fixed locus, of type I or II. 

\begin{notn} 
Let $\pi:S \times \CT\to \CT$ denote the canonical projection
and let $\mathbf{R}\mathscr{H}om_\pi$ denote the derived functor  $\dR_{\pi*}\mathbf{R}\mathscr{H}om$. 
Moreover, the pull-back of any sheaf $M$ on $S$ to $S\times \CT$ 
will be denoted again by $M$ for simplicity. The distinction 
will be clear from the context. 
\end{notn} 

For type I components, let $\IE$ denote the universal sheaf on 
$S\times \CM_h(S, \gamma)$. Note also that $\CM_h(S, \gamma)$ 
has a natural trace-free 
virtual cycle $[\CM_h(S, \gamma)]^{vir}_0$, as shown 
in \cite[Theorem 4.1]{HT10}. 
Then Proposition 3.2 and Corollary 3.3 of \cite{GSY17b} prove: 
\be\label{eq:typeIB} 
DT_h(X, \gamma)_I = \int_{[\CM_h(S,\gamma)]^{\vir}_0}\frac{1}{e\left(\dR \hom_{\pi}(\eE, \eE\otimes L \cdot \t)\right)},\\
\ee
where $L\cdot {\bf t}$ denotes $L \otimes {\bf t}$ for simplicity. 

In order to write down the type II contributions, note that for any pair $(\alpha_1, \alpha_2)$ and for any $(n_1, n_2)$ the nested Hilbert scheme $S_{(\alpha_1, \alpha_2)}^{[n_1,n_2]}$
is naturally isomorphic to the nested Hilbert scheme 
$S_{(0, \alpha_2-\alpha_1)}^{[n_1, n_2]}$ since $H^1(\CO_S)=0$. 
The latter will be denoted below by $S_{\alpha}^{[n_1, n_2]}$, where $\alpha = \alpha_2-\alpha_1$. For $n_1=0$ this reduces 
to the Hilbert scheme of one dimensional subschemes $Z\subset S$ 
with 
\[
c_1(\CO_Z)=\alpha,\qquad c_2(\CO_Z) = n_2. 
\]
For $n_2=0$ and $\alpha=0$, one obtains the usual Hilbert scheme 
$S^{[n_1]}$ of $n_1$ points on $S$.
Moreover note that there are natural projections 
\[
S_\alpha^{[n_1, n_2]} \to S^{[n_1]}, \qquad S_\alpha^{[n_1, n_2]} \to S^{[n_2]}_\alpha. 
\]
Therefore any equivariant sheaf or complex of sheaves $\IF$ 
on $S^{[n_i]}$ or $S\times S^{[n_i]}$, $1\leq i \leq n$ has a  
canonical pull-back to $S_\alpha^{[n_1,n_2]}$ or $S\times S_\alpha^{[n_1,n_2]}$ respectively. For simplicity the pull-back will be denoted again by $\IF$. In particular, using these conventions, the universal flag on $S\times S_\alpha^{[n_1, n_2]}$
will be given by $\CI^{[n_1]} \subset \CI^{[n_2]}_{\alpha}$ where 
$\CI^{[n_1]}$, $\CI^{[n_2]}_{\alpha}$ are obtained by pulling back the universal objects from $S\times S^{[n_1]}$ and 
$S\times S^{[n_2]}_\alpha$ respectively.

Employing the above conventions, for any line bundle $M$ on $S$ and any non-zero integer $a\in \IZ$, $a\neq 0$,  let 
\[
\CQ_\alpha^{[n_1,n_2]}(M,a) = {(a{\bf s})^{\chi(M)}\over e(\mathbf{R}\mathscr{H}om_\pi(
\CI^{[n_1]},\CI_{\alpha}^{[n_2]}\otimes M\cdot {\bf t}^{-a})}.
\]
Again, $M \cdot {\bf t}^{-a}$ denotes $M \otimes {\bf t}^{-a}$ 
while ${\bf s}={c}_1({\bf t})$ is the equivariant 
first Chern class of ${\bf T}$, commonly referred to as the equivariant parameter. 
For any integer $n\geq 0$ let also  
${\sf T}_{S^{[n]}}$ denote the tangent bundle to the Hilbert scheme $S^{[n]}$. 
Moreover, as shown in 
\cite[Theorem 1]{GSY17a}, note that the nested Hilbert scheme has an intrinsic virtual cycle 
$[S_{\alpha}^{[n_1, n_2]}]^{vir}$. 
Then the formula 
derived in Proposition 3.11 and Corollary 3.13 in \cite{GSY17b}
for the residual contribution of a component of the fixed locus  
$\CT \cong S_\alpha^{[n_1, n_2]}$ 
reads 
\be\label{eq:typeIIB}
\bal  
DT_{\alpha}^{[n_1,n_2]} =  & {(-1)^{-c_1(L)\cdot \alpha + c_1(L)\cdot c_1(S)/2 + 3c_1(L)^2 /2} \over 2^{\chi(L^2)} (-{\bf s})^{\chi(L^2) + \chi(L) -\chi(L^{-1})} }\\
&  \int_{[S_{\alpha}^{[n_1, n_2]}]^{vir}} \big(\,
e({\sf T}_{S^{[n_1]}}\otimes L\cdot \t) \, e({\sf T}_{S^{[n_2]}}\otimes L\cdot \t)
\,\CQ_\alpha^{[n_1,n_2]}(K_S\otimes L^{-1},-1) \\
& \ \CQ_\alpha^{[n_1,n_2]}(K_S\otimes L^{-1},-1) 
\CQ_\alpha^{[n_1,n_2]}(L^{-1}, -1) \, \CQ_\alpha^{[n_1,n_2]}(K_2\otimes L^{-2}, -2)^{-1}\, \big).
\eal 
\ee

\subsection{Universality results}\label{univsect}

As shown in \cite[Proposition 4.4]{GSY17b}, using Muchizuki's wallcrossing formula \cite[Theorem~1.4.6]{M02} the type I contributions in \eqref{eq:typeIB} are expressed in terms of Seiberg-Witten invariants coupled with certain combinatorial coefficients. The resulting formula is reviewed below. 

Mochizuki's combinatorial coefficients are written in terms of 
equivariant integrals on products of Hilbert schemes of points on $S$. The equivariant structure is defined with respect to a torus 
${\bf T}'= \IC^\times$ which acts trivially on $S$ and its Hilbert schemes. One should make a clear distinction between ${\bf T}'$ 
and the torus ${\bf T}$ used above in virtual localization computations, 
which are completely unrelated. In fact the ${\bf T}'$ action has its origin in a torus action on the master space used in the proof 
of \cite[Theorem 1.4.6]{M02}, which is not manifestly used in this paper. Let  $\t'$ denote the trivial line bundle on $S$ with the $\C^\times$-action of weight 1 on the fibers and let $\s':=c_1(\t')$. Below let $\pi'$ denote the projection
\[ 
 S\times S^{[n_1]}\times S^{[n_2]}\to S^{[n_1]}\times S^{[n_2]}.
\]
Then for any $\alpha\in H_2(S,\IZ)$ let $L_\alpha$ be a line bundle on $S$ with $c_1(L_\alpha) = \alpha$ and let 
\begin{align}\label{eq:A_term}
\v{n_i}_{L_{\alpha}}:=
\pi'_{\ast} \left(\CO_{\z{n_i}} \otimes L_{\alpha} \right).
\end{align}
Note that $L_\alpha$ and hence $\v{n_i}_{L_{\alpha}}$ are uniquely 
determined by $\alpha$ up to isomorphism since $H^1(\CO_S)=0$. 
Moreover, for any equivariant sheaf $\CE$ on $S\times S^{[n_1]} \times S^{[n_2]}$ let 
\begin{align*}
&\sP(\CE) = e(-\dR \hom_{\pi'}(\CE,\CE\otimes L\cdot \t))
\end{align*}
Then, following Mochizuki, for any pair of effective 
curves classes $(\beta_1, \beta_2)\in H^2(S, \IZ)^{\oplus 2}$ such that $\beta = \beta_1+\beta_2$ let the coefficients 
$\sA(\gamma, \beta_1, \beta_2)$ be defined by 
\begin{align*} 
&\sA(\gamma, \beta_1, \beta_2) := \\ \notag
& \sum_{\begin{subarray}{c}
n_1+n_2=\\ n-\beta_1\cdot \beta_2 
\end{subarray}}
\int_{S^{[n_1]} \times S^{[n_2]}} \mathrm{Res}_{\s'=0}
\left( \frac{e\left(\v{n_1}_{L_{\beta_1}}\right)\cdot \sP \left(\i{n_1}_{L_{\beta_1}}\cdot \t'^{-1} \oplus \i{n_2}_{L_{\beta_2}}\cdot \t' \right)\cdot e\left(\v{n_2}_{L_{\beta_2}}\cdot \t'^2\right)}{(2s')^{n_1 + n_2 -p_g}\cdot \sQ\left(\i{n_1}_{L_{\beta_1}}\cdot \t'^{-1}, \i{n_2}_{L_{\beta_2}}\cdot \t'\right)}\right). 
\end{align*}
where
\begin{align*}
&\sQ\left(\i{n_1}_{L_{\beta_1}}\cdot \t'^{-1}, \i{n_2}_{L_{\beta_2}}\cdot \t'\right)=\\&e\left(
-\dR \hom_{\pi'}\left(\i{n_1}_{L_{\beta_1}}\cdot \t'^{-1}, \i{n_2}_{L_{\beta_2}}\cdot \t'\right)
- \dR \hom_{\pi'}\left(\i{n_2}_{L_{\beta_2}}\cdot \t', \i{n_1}_{L_{\beta_1}}\cdot \t'^{-1}\right)\right).
\end{align*} 
Then the following formula is proven in \cite[Proposition 4.4]{GSY17b}:
\begin{equation}
\DT_h(X, \gamma)_{I}=-\sum_{\begin{subarray}{c}
\beta_1 + \beta_2 =\beta \\  
\beta_1\cdot h < \beta_2 \cdot h
\end{subarray}} 
\mathrm{SW}(\gamma_1) \cdot 2^{2-\chi(v)} \cdot \sA(\gamma_1, \gamma_2, v; \sP_{1}\cup \alpha). 
\end{equation}
where the sum in the right hand side is over all decompositions 
$\beta= \beta_1+\beta_2$ with $\beta_1, \beta_2$ effective curve classes.

By analogy with a result of  G{\"o}ttsche and Kool \cite[Proposition 3.3]{GK17} using  \cite[Lemma 5.5]{GNY} and an adaptation of \cite{EGL} one has the following universality 
statement. 
\begin{prop}\label{univI}
The expression \eqref{eq:A_term} is a universal polynomial 
in the topological invariants 
\[ 
\beta_i^2, \quad \beta_i\cdot c_1(S), \quad \beta_i \cdot D, \quad \beta_1\cdot \beta_2, \quad D^2,\quad D\cdot c_1(S),\quad c_1(S)^2, \quad c_2(S). 
\]
with $1\leq i \leq 2$. 
\end{prop} 

An analogous universality result for the type II contributions 
\eqref{eq:typeIIB} is not readily available, except for type II 
components associated to 
nested Hilbert scheme of points with no divisorial twists, i.e. $\S{n_1\geq n_2}:= \S{n_1,n_2}_\alpha$ with $\alpha=0$.
In this case, 
using \cite[Theorem 6]{GSY17a} the right hand side of \eqref{eq:typeIIB} simplifies as follows. 

For any pair $(n_1, n_2)\in \IZ^2$, $n_1, n_2\geq 0$, note that there is a natural closed embedding $\iota: S^{[n_1,n_2]} \to 
S^{[n_1]} \times S^{[n_2]}$ and a cartesian square 
\be\label{eq:Hschemesquare} 
\xymatrix{ 
S \times S^{[n_1,n_2]} \ar[r]^-{\iota'} \ar[d]^-{\pi} & 
S \times S^{[n_1]} \times S^{[n_2]}\ar[d]^-{\pi'} \\ 
S^{[n_1,n_2]} \ar[r]^-{\iota} &  S^{[n_1]} \times S^{[n_2]},\\}
\ee
where $\pi, \pi'$ are the natural projections. Then, as in 
\cite[Definition 3.6]{GSY17b} for any 
line bundle $M$ on $S$, let 
\[ 
\sE^{n_1, n_2}_{M} = [\dR_{\pi'*} M] -[\dR\hom_{\pi'}(\CI^{[n_1]}, \CI^{[n_2]}\otimes M )].
\]
Then the following holds 
\begin{lem}\label{simplifiedII} 
The equivariant residual contribution of a Type II component 
$\CT \cong \S{n_1,n_2}$ 
is given by 
\be\label{eq:product}
\begin{aligned} 
\DT_0^{[n_1,n_2]} =\ & \frac{(-1)^{c_1(S)\cdot D/2+3D^2/2}}{2^{\chi(L^2)}(-\s)^{\chi(L^2)+\chi(L)-\chi(L^{-1})}}\\
& \int_{\S{n_1}\times \S{n_2}}
\frac { c_{n_1+n_2}(\sE^{n_1,n_2}) \cup 
e(\sT_{S^{[n_1]}} \otimes L\cdot \t)\,  e(\sT_{S^{[n_2]}} \otimes L\cdot \t) \, e(\sE^{n_1,n_2}_{K_{S}\otimes L^{-2}}\cdot \t^{-2})}{ e(\sE^{n_1,n_2}_{K_{S}\otimes L^{-1}}\cdot \t^{-1})\cdot e(\sE^{n_1,n_2}_{L^{-1}}\cdot \t^{-1})},\\
\end{aligned}
\ee
\end{lem} 

{\it Proof}. The $K$-theory class $\sE_{M\cdot \t^a}^{n_1,n_2}$ satisfies the pull-back property
\be\label{eq:pullbackprop}
\iota^*\sE_{M\cdot \t^a}^{n_1,n_2}\cong [\dR\pi_* M] - [\dR\pi_*(\CI^{[n_1]}, \CI^{[n_2]}\otimes M\cdot \t^a)] 
\ee
since the universal sheaves $\CI^{[n_i]}$ on $S^{[n_1,n_2]}$ are 
naturally isomorphic to the restriction of the analogous objects from
$S^{[n_1]}\times S^{[n_2]}$. 
This implies that 
\[ 
\iota^*e(\sE_{M\cdot \t^a}^{n_1,n_2}) = e(\CQ_0^{[n_1,n_2]}(M, a))
\]
for any $M$ and any $a\in \IZ$. Furthermore, as shown in 
\cite[Lemma 5.4]{GSY17a}, the equivariant class $\CQ_0^{[n_1,n_2]}(M, a)$ is well behaved under good degenerations on $S$. For completeness, note that 
the right hand side of \eqref{eq:pullbackprop} was denoted by 
$\sK^{[n_1\geq n_2]}_{M \cdot \t^a}$ in \cite[Lemma 5.4]{GSY17a}, 
assuming without loss of generality that $n_1\geq n_2$.  Moreover, the good degeneration property is explicitly stated in \cite[Remark 5.10]{GSY17a}. 

In conclusion Theorem 6 in  \cite{GSY17a} applies to the present situation, leading to equation \eqref{eq:product}. 

\hfill $\Box$ 

To conclude, applying \cite[Lemma 5.5]{GNY} to our situation, \cite[Corollary 5.9, Remark 5.10, Proposition 5.11, Appendix 5.A.]{GSY17a} yield the following 
\begin{prop}\label{univII}
Each type II contribution $DT_{0}^{[n_1,n_2]}$ is a universal polynomial 
in the topological invariants 
\[ 
\beta_i^2, \quad \beta_i\cdot c_1(S), \quad \beta_i \cdot D, \quad \beta_1\cdot \beta_2, \quad D^2,\quad D\cdot c_1(S),\quad c_1(S)^2, \quad c_2(S). 
\]
with $1\leq i \leq 2$. 
\end{prop}

\section{Rank two calculations for local elliptic K3 surfaces}\label{kthreesect} 

The goal of this section is to carry out explicit computations and 
formulate a modularity conjecture for rank two invariants in a specific class of examples. Namely, the surface $S$ in Section \ref{DTsect} will be a smooth generic elliptic $K3$ surface in Weierstrass form and the divisor $D$ will be a positive multiple of the elliptic fiber class. 
 
\subsection{Chamber structure for local elliptic surfaces}\label{chambersect} 

 In this section $S$ will be a smooth generic elliptic surface in Weierstrass form over 
the complex projective line $C=\IP^1$. It will be assumed 
that all singular fibers of $S$ over $C$ are nodal elliptic curves, and there is at least one such fiber. In particular all fibers are irreducible and there are no multiple fibers. 
Under these assumptions \cite[Thm. 24, Ch. 7]{Vb_elliptic} shows that $S$ has trivial fundamental group, while Proposition 27, Ch. 7 in loc. cit. shows that $H^2(X,\IZ)$ is torsion free. In particular $H^1(S, \IZ)=0$, hence $H^1(\CO_S)=0$ as well. 

Furthermore it will be assumed throughout this section that the Mordell-Weil group of $S$ is trivial. 
This implies that the Picard group of $S$ is freely generated by $(f,\sigma)$, where $\mathscr{F}$ is the elliptic fiber class. Moreover, there is a natural isomorphism ${\rm Pic}(S) \cong H_2(S,\IZ)$ which will be implicitly used below. 
Finally, \cite[Thm. 15, Ch. 7]{Vb_elliptic} shows that the canonical line bundle of such a surface is isomorphic to $K_S \cong \CO_S(kf)$ for some $k \in \IZ$, $k \geq -1$. In this section it will be assumed that $k\geq 0$. Since the canonical section is isomorphic to the projective line, a simple computation yields the intersection numbers:
\be\label{eq:Sinters} 
\sigma^2 = -k-2, \qquad f^2 =0, \qquad \sigma \cdot f =1.
\ee
Moreover, note the following. 
\begin{lem}\label{Scone} 
Under the above assumptions, 

$(i)$ The cone effective divisors 
in ${\rm Pic}(S)$ is generated by $(\sigma, f)$, and 

$(ii)$ 
A real divisor class $h=t \sigma + uf$ in ${\rm Pic}_{\IR}(S)$, with $t,u\in \IR$ is ample if and only if $0<(k+2)t<u$.  
\end{lem} 

{\it Proof}.  Suppose 
$\Sigma\subset S$ is an nonempty irreducible curve on $S$ and let 
$[\Sigma]\in H_2(S,\IZ)$ denote the associated curve class. 
Under the current assumptions, ${\rm Pic}(S) \cong H_2(S,\IZ)$ is freely generated by $(\sigma, f)$. Therefore $[\Sigma]= a\sigma+ bf$ with $a,b\in \IZ$. Since $\Sigma$ is assumed irreducible, it can be either a multiple of a fiber 
of the projection $\pi:S\to C$, or an irreducible finite-to-one cover of $C$. In both cases, note that $[\Sigma]\cdot f \geq 0$, hence $a\geq 0$. 
If $a=0$, it follows that $b>0$ since $\Sigma$ is assumed effective, non-empty. 

Suppose $a>0$. 
Since $\Sigma$ is irreducible, and the Mordell-Weil group of $S$ is trivial by assumption, 
one of the following cases must hold: 
\begin{itemize} 
\item[$(a)$] $\Sigma$ is set theoretically supported on the 
canonical section, in which case $b=0$, or 
\item[$(b)$] $\Sigma$ is not set theoretically supported on the 
canonical section, in which case $[\Sigma]\cdot \sigma >0$. Using 
\eqref{eq:Sinters} this implies $b> a(k+2)>0$. 
\end{itemize} 

In conclusion, any effective irreducible curve class must be of the form $a\sigma + b f$ with $a, b \geq 0$. Hence the same holds 
for an arbitrary effective curve class, proving $(i)$. 
The second claim follows immediately from the intersection 
table \eqref{eq:Sinters}.

\hfill $\Box$ 

Now, in the framework of Section \ref{DTsect}, let $S$ be as above, and let $D=mf$ with $m\geq 1$. Hence the local fourfold $X$ is the total space of 
the direct sum of line bundles 
$K_S(mf) \oplus \CO_S(-mf) \cong \CO_S((k+m)f) \oplus 
\CO_S(-mf)$. Since $m\geq 1$ by assumption, Lemma \ref{supplemmaA} 
implies that any $h$-stable compactly supported pure dimension two sheaf 
on $X$ is scheme theoretically supported on $Y$. The main goal of the section is to prove a global boundedness result for such sheaves with fixed topological invariants. 
This will require a Bogomolov type inequality, which will be derived below from the following categorical equivalence proven in \cite[Proposition 2.2]{TT17a}. 

Let $\Higgs_L(S)$ denote the abelian category of coherent Higgs sheaves on $(\mathscr{E}, \phi)$ on $S$, where 
$\mathscr{E}$ is a coherent sheaf on $S$ and $\phi: \mathscr{E} \to \mathscr{E}\otimes L$ a morphism of $\CO_S$-modules. Let $\Coh_c(Y)$ denote the abelian category of compactly supported coherent sheaves on $Y$. Recall that $\pi_S:Y\to S$ denotes the natural projection and $\zeta: H^0(Y,\pi_S^*L)$ is the tautological section. Then note that any  object $\mathscr{F}$ in $\Coh_c(Y)$ determines a Higgs sheaf $(\mathscr{E},\phi)$ on $S$, where 
\begin{itemize} 
\item $\mathscr{E}= \pi_{S*}\mathscr{F}$, and 
\item $\phi: \CE \to \CE \otimes L$ is the pushforward of the natural map $$\zeta \otimes {\bf 1}_\mathscr{F} : \mathscr{F} \otimes \pi_S^*L^{-1} \to \mathscr{F}.$$
\end{itemize} 
Then the following holds (\cite[Proposition 2.2]{TT17a}).\begin{prop}\label{higgsequiv} 
There is an equivalence of categories $\Coh_c(Y) {\buildrel \sim \over \longto} \Higgs_L(S)$ which maps $\mathscr{F}$ to the pair $(\mathscr{E},\phi)$ 
constructed above. 
\end{prop} 

Recall that the topological invariants of a compactly supported 
sheaf $\mathscr{F}$ on $X$ where defined in Section \ref{DTsect} as 
the Chen classes of $\mathscr{E}= g_*\mathscr{F}$ encoded in the triple 
\[
\gamma(F) = (r(\mathscr{F}), \beta(\mathscr{F}), n(\mathscr{F})) \in \IZ \oplus H_2(S,\IZ)\oplus \IZ. 
\]
Then note the following consequence of Proposition \ref{higgsequiv}, proven in \cite[Lemma 2.9]{TT17a}:
\begin{cor}\label{higgscor} 
Let $h \in {\rm Pic}_\IR(S)$ be an ample real divisor class on $S$. Let $\mathscr{F}$ be a compactly supported sheaf on $Y$ with 
$r(\mathscr{F})>0$. 
Then $\mathscr{F}$ is $h$-(semi)stable if and  only if the associated Higgs sheaf $(\mathscr{E},\phi)$ is $h$-(semi)stable. 
\end{cor}

Next note that the Bogomolov inequality holds for slope semistable Higgs 
sheaves on $S$. As shown below, this follows from 
\cite[Thm. 4.3]{HK_quivers}, which is proven analytically. 
In order to show that the conditions of loc. cit  are satisfied, one first needs the following. 

\begin{lem}\label{kahlerlemma} Let $h=t\sigma+uf$ a real ample class 
on $S$ with $0<(k+2)t<u$. Let $\omega_C$ be the standard Fubiny-Study K\"ahler form on $C=\IP^1$. Then there exists a 
K\"ahler form $\omega_h$ such that the Hodge class 
$[\omega_h]\in H^{1,1}(S,\IR)$ is Poincar\'e dual to $h$, 
and $\omega_h|_C = (u-(k+2)t) \omega_C$. 
\end{lem} 

\begin{proof} Since $S$ is assumed to be in Weierstrass form, it 
admits a presentation as a hypersurface of bi-degree $(6(k+2),3)$ in the toric 
variety 
\[
Z= \left(\IC^2\times \IC^3 \setminus (\IC^2 \times \{0\} \cup \{0\} \times \IC^3)\right)/\IC^\times \times \IC^\times
\] 
where the $\IC^\times \times \IC^\times$ action is given by 
\[
\begin{array}{ccccc}  
u & v & x & y & z \\ 
1 & 1 & 2(k+2) & 3(k+2) & 0 \\
0 & 0 & 1 & 1 & 1 \\ 
\end{array} 
\]
The divisor classes $\sigma, f\in {\rm Pic}(S)$ are obtained
by restriction of toric divisors, 
\[
\sigma = Z|_{S}, \qquad f = U|_{S}
\]
where $U,Z$ are the toric divisors defined by $u=0$ and $z=0$ respectively. These are the same as the standard generators of the Picard group of $Z$ associated to the two $\IC^\times$ factors.  Then the claim follows from the correspondence between toric and symplectic K\"ahler quotients. 
\end{proof}
The next result needed in the following is the existence of a suitable hermitian structure on $L$. 
\begin{lem}\label{positiveF} 
Let $\omega_h$ be a K\"ahler form on $S$ as in Lemma \ref{kahlerlemma}. Then there exists a hermitian structure on $L$ and a hermitian connection with curvature $\gamma$ such that the contraction $\Lambda_{\omega_h} \gamma$ is pointwise positive semidefinite. 
\end{lem} 

\begin{proof}
Note that $L$ is the pull-back of the line bundle $\CO_C(m+k)$. The latter admits a hermitian structure and a hermitian connection with curvature $(m+k)\omega_C$. The required structure on $L$ is obtained by pull-back. 
\end{proof}

As usual for any rank $r>0$ torsion free 
Higgs sheaf $(\mathscr{E},\phi)$ on $S$ let 
\[ 
\Delta(\mathscr{E}) = c_2(\mathscr{E})-{r-1\over 2r} c_1(\mathscr{E})^2. 
\]
Then the following holds. 
\begin{lem}\label{bogomolov} Let $(\mathscr{E}, \phi)$ be a $\mu_h$-semistable rank $r>0$ torsion free Higgs sheaf on $S$. Then 
\[ 
\Delta(\mathscr{E})\geq 0. 
\]
\end{lem}

\begin{proof}
First suppose $E$ is locally free. Then Lemmas 
\ref{kahlerlemma} and \ref{positiveF} show that the conditions of  \cite[Thm. 4.3]{HK_quivers} are satisfied in 
the present situation. This implies the result in this case. 
If $E$ is torsion free but not locally free, the double dual 
Higgs sheaf 
$(\mathscr{E}^{\vee\vee}, \phi^{\vee\vee})$ is $\mu_h$-semistable and locally free. Moreover, 
$c_2(\mathscr{E}) - c_2(\mathscr{E}^{\vee\vee}) = \chi(\mathscr{E}^{\vee\vee}/\mathscr{E}) >0$
while $c_1(\mathscr{E})=c_1(\mathscr{E}^{\vee\vee})$.  
\end{proof}
Using Lemma \ref{bogomolov}, one proves the following 
global boundedness result. 
Let $\mathscr{F}$ be a $\mu_h$-stable compactly supported sheaf on $Y$ with invariants  
\be\label{eq:topinvC} 
\gamma(F) = (r, \beta, n) \in \IZ \oplus H_2(S,\IZ)\oplus \IZ. 
\ee
Suppose $r>0$. 
According to Corollary \ref{higgscor}, the associated Higgs bundle on $S$ is also $\mu_h$-stable, hence $n \geq 0$ by Lemma \ref{bogomolov}. 
Then one has: 
\begin{lem}\label{nowalls} 
Suppose there exists $t'\in \IR$, $0<t'<t$, such that $E$ is not $\mu_{h'}$-semistable for $h'=t'\sigma + uf$. Then 
\[
{t\over u} > {2\over k+2 + 2 r^3\Delta(E)}
\]
\end{lem}
 
\begin{proof}
Completely analogous to the proof of
\cite[Propposition 6.2]{chamber} and 
\cite[Lemma 4.7]{FM_vertical}. Uses the Bogomolov inequality and the algebraic Hodge index theorem on $S$. 
\end{proof}

The main consequence of Lemma \ref{nowalls} is the absence of marginal stability walls for sheaves $\mathscr{F}$ with fixed topological invariants \eqref{eq:topinvC}  in the open chamber 
\[
{t\over u} < {2\over  k+2 + 2 r^3\Delta(E)}.
\]
By analogy with \cite[Lemma 1.2]{notes_elliptic} and also \cite[ Proposition 1.2]{FM_vertical} this implies 
a generic stability result for such sheaves. For completeness, this will be formulated and proved below in detail. 
Several intermediate steps are necessary.  

Let $\delta \subset C$ be the discriminant of the elliptic fibration, which, under the current assumptions, consists of $12(k+2)$ pairwise distinct closed points. 

\begin{lem}\label{restrlemma} 
Let $\mathscr{F}$ be a compactly supported  sheaf on $Y$ with topological invariants \eqref{eq:topinvC},
where $r>0$.   
Suppose there exists an open subset $V\subset C\setminus \delta$ such that 
$\mathscr{F}_V = \mathscr{F}\otimes q^*\CO_{V}$ has pure dimension two.  Then for any point $p\in V$ there is an exact sequence 
\[ 
0 \to \mathscr{F}\otimes \CO_{Y}(-Y_p) \to \mathscr{F} \to \mathscr{F}_p\to 0
\]
where $Y_p\subset Y$ is the fiber $Y_p = \pi_S^{-1}(p)$, and $\mathscr{F}_p = \mathscr{F} \otimes \CO_{Y_p}$. Moreover, 
\[ 
\ch_0(\mathscr{F}_p)=0, \qquad \ch_1(\mathscr{F}_p) = r f, \qquad 
\ch_2(\mathscr{F}_p) = \beta\cdot f.
\]
\end{lem}

{\it Proof.} Under the given assumptions, the scheme theoretic support of $\mathscr{F}$ is a proper horizontal divisor 
in $Y$ i.e. a compact finite-to-one cover of $S$ which intersects $Y_p$ along a curve. In particular the local tor sheaf ${\mathcal Tor}_1^{Y}(\CO_{Y_p}, \mathscr{F} )$ is set theoretically supported in dimension one, hence the natural map 
${\mathcal Tor}_1^{Y}(\CO_{T_p}, \mathscr{F} )\to  \mathscr{F} $ must be identically zero for any $p\in V$, by the purity assumption for $\mathscr{F}$. This proves exactness. The Chern character follows by a straightforward computation in the intersection ring of $Y$. 

\hfill $\Box$

\begin{rmk}\label{stabrem} 
In the framework of Lemma \ref{restrlemma} note that 
$Y_p \cong S_p \times \IA_1$ is an elliptic surface, and 
$\mathscr{F}_p$ is set theoretically supported on a finite set of elliptic fibers $Y_p$ over $\IA^1$. In particular it has compact support. In the following a compactly supported sheaf $\mathcal{G}$ on $Y_p$ with 
invariants 
\[ 
\ch_0(\mathcal{G})=0, \qquad \ch_1(\mathcal{G}) = r_\mathcal{G} f
\] 
will be called (semi)stable if and only if 
\[
\ch_2(\mathcal{G}')/r_{\mathcal{G}'} \ (\leq)\ \ch_2(\mathcal{G})/r_\mathcal{G}
\]
for any proper nontrivial subsheaf $0\subset \mathcal{G}' \subset \mathcal{G}$ with invariants 
\[ 
\ch_0(\mathcal{G})=0, \qquad \ch_1(\mathcal{G}) = r_{\mathcal{G}'} f, \quad r_{\mathcal{G}'}>0. 
\]
\end{rmk}

Next note the following lemma, which follows from 
Theorems 2.1.5 and 2.3.2 in \cite{HL97}.

\begin{lem}\label{HNlemma} 
Let $\mathscr{F}$ be a compactly supported pure dimension two sheaf on $Y$ with invariants $(r,\beta, n)$, $r>0$. 
Then there exist an open subset $V\subset C \setminus \delta$ and a relative Harder-Narasimhan filtration 
\be\label{eq:openHNa}
0 \subset \mathscr{F}_{1}  \subset \cdots \subset \mathscr{F}_l = \mathscr{F}_V
\ee
such that all successive quotients are flat over $V$. 
\end{lem}

Using Lemma \ref{HNlemma}, one next proves:

\begin{lem}\label{genstab} 
Let $\mathscr{F}$ be a  compactly supported sheaf on $Y$ of pure dimension two with invariants  $(r,\ell\, f, n)$, $r>0$, $\ell \in \IZ$.  Suppose $\mathscr{F}$ is $\mu_h$-semistable and 
\be\label{eq:smallchamber} 
{t\over u} < {2\over k+2 + 2 nr^3}.
\ee
Then there exists an open subset $V\subset C\setminus \delta$ such that 
the restriction $\mathscr{F}_p=\mathscr{F}|_{Y_p}$ is semistable for any closed point $p\in V$.
\end{lem} 

{\it Proof}. 
This statement is analogous to 
\cite[Lemma 1.2]{notes_elliptic} and also \cite[ Proposition 1.2]{FM_vertical}. The proof is given below for completeness. 

Let $\mathscr{F}$ be a sheaf as in Lemma \ref{genstab} and let $V\subset C\setminus \delta$ be the open subset of 
Lemma \ref{HNlemma}.  By construction the succesive quotients of the fitration \eqref{eq:openHNa} are flat over $V$ and the filtration restricts to a Harder-Narasimhan filtration for $\mathscr{F}_p$ on each fiber $Y_p$, $p\in V$. Using \cite[Ex. 5.15(d)]{H77}, this filtration extends to a filtration of sheaves 
\[
0 \subset {\oF}_1 \subset \cdots \subset \oF_l= \mathscr{F}
\]
on $Y$. Let 
\[
\gamma(\oF_1) = (r_1, \beta_1, n_1) \in \IZ \oplus H_2(S,\IZ) \oplus \IZ.
\]
be the invariants of $\oF_1$. By construction, $r_1>0$.  Using Lemma \ref{restrlemma}, one finds 
\[ 
\ch_0(\oF_{1,p}) =0, \qquad \ch_1(\oF_{1,p}) = r_1 f, 
\qquad \ch_2(\oF_{1,p}) = \beta_1\cdot f. 
\]
At the same time 
\[
\ch_0(\oF_{p}) =0, \qquad \ch_1(\oF_{p}) = r f, 
\qquad \ch_2(\oF_{p}) = 0.  
\]
Therefore one must have 
\[
\beta_1 \cdot f \geq 0. 
\]
However, by assumption, $\mathscr{F}$ is $\mu_h$-slope semistable 
for any $t>0$ satisfying inequality \eqref{eq:smallchamber}. 
This implies that 
\[ 
\beta_1 \cdot f \leq 0
\]
Therefore $\beta_1\cdot f=0$, which implies that 
$\mathscr{F}_p$ is semistable. 

\hfill $\Box$

\subsection{Type II fixed loci for elliptic K3 surfaces} 
Lemma \ref{genstab} 
has important consequences for the structure of 
type II fixed loci, assuming $S$ to be an elliptic K3 surface. 
In this section the rank will be fixed, $r=2$. 
Recall that each connected component of the type II fixed locus is isomorphic to a nested Hilbert scheme $S_{(\beta_1,\beta_2+c_1(L))}^{[n_1,n_2]}$ parameterizing flags of twisted ideal sheaves $\CI_{1}\otimes L_{1}\subset \CI_{2}\otimes L_2\otimes L$. Here 
$\CI_i$, $1\leq i \leq 2$ are ideal sheaves of zero dimensional subschemes of $S$ and $L_1, L_2$ are line bundles on $S$ such that 
\[
c_2(\CI_i) = n_i, \qquad c_1(L_i)=\beta_i, \qquad 1\leq i \leq 2. 
\]
Moreover, the topological invariants $(n_1, n_2)$, $\beta_1, \beta_2$ are subject to the conditions specified in equation 
\eqref{eq:typeII} as well as $(II.a)$ and $(II.b)$ below 
\eqref{eq:typeII}. 
In particular the 
curve class $\alpha = \beta_2+c_1(L)_\beta-1$ 
must be effective. 
As shown in \cite[Section 3.1]{GSY17b}, the associated stable sheaf $\mathscr{F}$ on $Y$ is then an extension of $\CO_Y$-modules 
\be\label{eq:sheafextA}
0 \to \CI_{2}\otimes L_2\otimes L\to \mathscr{F} \to \CI_{1}\otimes L_1 \to 0. 
\ee

For completeness note that this also follows from Proposition 2.2 in \cite{TT17a}. Denoting the injection 
$\CI_{1}\otimes L_1 \subset \CI_{2}\otimes L_2\otimes L$ by $\varphi$, each such triple 
determines a Higgs sheaf $(\mathscr{E},\phi)$ on $S$ with 
\[ 
\mathscr{E} = \CI_{1}\otimes L_1 \oplus \CI_{2}\otimes L_2, \qquad 
\phi=\left(\begin{array}{cc} 0 & 0 \\ \varphi & 0 \end{array}\right).
\]
Then \cite[Proposition 2.2]{TT17a} implies that the associated sheaf on $\mathscr{F}$ must be an extension of the form \eqref{eq:sheafextA}.

For an elliptic surface as above, each class $\beta_i$ is of the form $\beta_i = a_i \sigma + b_i f$, $1\leq i\leq 2$. 
Then Lemma \ref{genstab} yields the following. 
\begin{lem}\label{adiabstab} 
Let $\mathscr{F}$ be an extension of the form \eqref{eq:sheafextA} corresponding to a flag of twisted ideal sheaves $\CI_{1}\otimes L_1 
\subset \CI_{2}\otimes L_2 \otimes L$ with $\alpha = \beta_2+c_1(L)-\beta_1$ effective. 
Suppose $\mathscr{F}$ is $\mu_h$-semistable, $a_1+a_2=0$,  and 
\be\label{eq:smallchamber} 
{t\over u} < {2\over k+2 + 16n}.
\ee
Then $a_1=a_2=0$. 
\end{lem}

{\it Proof}. 
As shown in Lemma \ref{genstab}, under the given conditions, $\mathscr{F}$ has to be $\mu_h$-semistable for all polarizations $h'= t'\sigma + uf$, with 
$0<t'<t$.  This implies $a_2\leq 0$. Since $\alpha$ is assumed effective, Lemma \ref{Scone} implies that $a_1=a_2=0$.  

\hfill $\Box$ 

\begin{lem}\label{oneoneKthree} 
Let $S$ be an elliptic K3 surface in Lemma \ref{adiabstab}. 
Let $\mathscr{F}$ be a $\mu_h$-semistable extension of the form \eqref{eq:sheafextA} with $\beta_1=bf$, $\beta_2 = (1-b)f$.
Moreover, suppose that 
\be\label{eq:smallchamber} 
{t\over u} < {1\over 8n+1}.
\ee
Then the following inequalities hold 
\be\label{eq:stabcondA} 
1\leq b \leq (m+1)/2.
\ee
\end{lem}

{\it Proof.} In this case $\alpha = (m+1-2b)f$. Since this must be effective by assumption, $2b \leq m+1$. Moreover, condition 
$(II.b)$ below equation \eqref{eq:typeII} implies $b \geq 1$. 

\hfill $\Box$ 

In conclusion, using equation \eqref{eq:typeII}, Lemma \ref{oneoneKthree} yields:  
\begin{lem}\label{adiabtypeII} 
Let $S$ be a smooth generic elliptic K3 surface in Weierstrass form and let $\gamma = (2, f, n)$ be a triple of topological invariants with $n\geq 0$. Suppose inequality \eqref{eq:smallchamber} is satisfied. Then the type II fixed locus is a union 
\be\label{eq:typeIIK3} 
\CM_h(X,\gamma)^{\bf T}_{II} = \bigcup_{1\leq b \leq (m+1)/2}\ \  \bigcup_{\substack{n_1,n_2\geq 0\\ n_1+n_2=n\\ n_1\geq n_2 \ {\rm if}\ b= (m+1)/2}} \ S^{[n_1,n_2]}_{(m+1-2b)f}\, .
\ee
\end{lem}

\subsection{Partition functions and modularity conjectures}\label{partfctsect}

In this section $S$ will be a smooth generic elliptic K3 surface in Weierstrass form. Recall that $X$ is the total space $K_S(D)$, 
which in this case is isomorphic to $\CO_S(mf)$, $m\geq 1$. 
For any triple $\gamma= (2, f, n)$, with $n\geq 0$, let 
\[ 
DT_\infty(X,\gamma)_I = DT_h(X, \gamma)_I, \qquad 
 DT_\infty(X,\gamma)_{II} = DT_h(X, \gamma)_{II},
 \]
for any polarization $h = t\sigma + uf$ satisfying inequality \eqref{eq:smallchamber}. Lemmas \ref{nowalls} and \ref{adiabtypeII} show that these 
invariants are independent of the choice of $h$ as long as inequality 
\eqref{eq:smallchamber} is satisfied. Then 
let
\[ 
Z_\infty(X, 2,f;q)_I = \sum_{n \geq 0} DT_\infty(X,2,f,n)_I q^{n-2} 
\]
and 
\[ 
Z_\infty(X, 2,f;q)_{II} = \sum_{n \geq 0} DT_\infty(X,2,f,n)_{II} q^{n-2} 
\]
the generating functions of such invariants. Note that a priori 
the above expressions are elements of $q^{-2} \IQ({\bf s})[[q]] \subset 
\IQ({\bf s})[[q^{-1}, q]]$. 
\bigskip

{\it Proof of Proposition \ref{DTkthreeI}}. Recall that 
Proposition \ref{DTkthreeI} claims that  $Z_\infty(X, 2,f;q)_I$ belongs to $\IQ({\bf s})[[q]]\subset \IQ({\bf s})[[q^{-1}, q]]$ and that the following identity holds in $\IQ({\bf s})[[q^{1/2}]]$: 
\be\label{eq:typeIformula} 
Z_\infty(X, 2,f;q)_I = {{\bf s}^{-1}\over 2}(\Delta^{-1}(q^{1/2}) + \Delta^{-1}(-q^{1/2}))  
\ee
In order to prove this note  that the type I contributions are given by equation \eqref{eq:typeIB}, which is reproduced below fro convenience: 
\[
DT_h(X, \gamma)_I = \int_{[\CM_h(S,\gamma)]^{\vir}_0}\frac{1}{e\left(\dR \hom_{\pi}(\eE, \eE\otimes L \cdot \t)\right)}.\\
\]
Under the current assumptions Proposition \ref{univI} 
shows that the virtual integral in the right hand side is independent of $L \cong \CO_S(mf)$ since 
\[ 
D^2=D\cdot c_1(S)= D\cdot \beta =0. 
\]
As a result, the above integral equals 
\[ 
DT_h(X, \gamma)_I = \int_{[\CM_h(S,\gamma)]^{\vir}_0}\frac{1}{e\left(\dR \hom_{\pi}(\eE, \eE\otimes \t)\right)}.\\
\]
As explained in \cite[Remark. 2.10]{GSY17b}, these invariants 
are related to the Vafa-Witten invariants ${\sf V}{\sf W}_h(S, \gamma)$ defined in \cite{TT17a} by 
\[
DT_h(X, \gamma)_I  = {\bf s}^{-1} {\sf V}{\sf W}_h(S, \gamma).
\]
Moreover, Proposition 7.4 of loc. cit. shows that under the present assumptions, the latter 
are given by the topological Euler character of the moduli space 
$\CM_h(S,\gamma)$, which is in this case smooth. 
In conclusion, one has 
\[ 
DT_h(X, \gamma)_I  = {\bf s}^{-1} \chi^{\rm top}(\CM_h(S,\gamma)).
\]
Finally, note that the moduli space $\CM_h(S,\gamma)$ is shown 
in \cite{Mukai84} and \cite[Section 6]{HL97} to be deformation equivalent to the Hilbert scheme $S^{[2n-3]}$. In particular the moduli space is empty for $0\leq n \leq 1$, which proves that indeed the series $Z_\infty(X, 2,f;q)_I$ belongs to $\IQ({\bf s})[[q]]\subset \IQ({\bf s})[[q^{-1}, q]]$. 
Moreover, identity \eqref{eq:typeIformula} follows by straightforward computations from Goettsche's formula.

\hfill $\Box$ 

\begin{prop}\label{DTkthreeII} 
Under the same conditions as in Lemma \ref{oneoneKthree}, the 
type II contributions $DT_\infty(X,\gamma)_{II}$ vanish for 
$m$ even, and reduced to 
\be\label{eq:typeIIformula} 
DT_\infty(X,\gamma)_{II} =\sum_{\substack{n_1\geq n_2\geq 0\\ n_1+n_2=n}} DT_0^{[n_1,n_2]}
\ee
for $m$ odd. 
\end{prop} 

{\it Proof}. 
Note that for $1\leq b \leq (m+1)/2$ one has 
$h^0(S,\CO_S((m+1-2b)f) >0$ and  
\[
h^2(S,\CO_S((m+1-2b)f)= \left\{\begin{array}{ll} 
0, & {\rm if}\ 2b < m+1, \\
1, & {\rm if}\ 2b = m+1.\\
\end{array}\right.
\]
This shows that under to given assumptions all type 
II fixed loci on $S$ satisfy the vanishing conditions of \cite[Thm. 7]{GSY17b} for $2b < m+1$, i.e. 
\[
h^0(S,\CO_S(m+1-2b)) >0, \qquad h^2(S,\CO_S(m+1-2b))=0
\]
If $2b=m+1$, one has $\alpha=0$, hence all type II components satisfying this condition are isomorphic to nested Hilbert schemes 
without divisorial twists. The condition $n_1\geq n_2$ comes from 
equation \eqref{eq:typeIIK3}.

\hfill $\Box$ 

In conclusion, for $m$ even, the partition function $Z_\infty(X,2,f;q)$ equals the generating function of type I contributions \eqref{eq:typeIformula}. For $m$ odd one obtains the additional type II contributions \eqref{eq:typeIIformula}. Each term in this formula is an equivariant residual integral of the form \eqref{eq:product}.
Again, Proposition 
\ref{univII} shows that under the present conditions, the integrals \eqref{eq:product} are independent of $m$ since 
$D^2 = D\cdot c_1(S)=D\cdot \beta =0$. Therefore each integral 
\eqref{eq:product} is in fact equal to its $m=0$ counterpart. 
There is no contradiction here, since the condition that $m$ be odd only occurred as a selection rule for type II contributions in  Lemma \ref{adiabtypeII} and Proposition \ref{DTkthreeII}. Once that is done, the resulting  
integrals \eqref{eq:product} can be computed in isolation since they are certainly well defined for any value of $m$, in particular $m=0$. 

Setting $m=0$ in \eqref{eq:product}, as explained in \cite[Remark 
2.10]{GSY17b}, the resulting invariants {\it formally} coincide 
up to an equivariant factor to the type II Vafa-Witten invariants defined and computed in Sections 
8.7-8.9 of \cite{TT17a}. It should be emphasized that while the resulting virtual integrals are formally the same,
the theoretical framework for these two computations is not the same. In the framework of this paper the type II contributions studied in \cite{TT17a} occur in the local fourfold theory with $L=K_S$ where $S$ is a general type surface with $c_1(S)\neq 0$ and primitive. In particular such contributions are ruled out by stability for local K3 surfaces with $L=K_S \cong \CO_S$ and for primitive invariants $\gamma$. 
The main point of the above argument is to point out that formally identical equivariant 
integrals also occur in the $L=K_S(mf)$-twisted theory with $S$ a K3 surface, at least if $m$ is odd. In the latter case they are 
allowed by the stability conditions, as shown in Lemma \ref{adiabtypeII}. 
In conclusion, while the context is different, the resulting contributions are formally identical up to a factor of ${\bf s}^{-1}$, which was explained in \cite[Remark 2.10]{GSY17b}. Therefore the computations in Sections 8.7-8.9 of \cite{TT17a} lead to Conjecture \ref{typeIIconj}.

 \appendix 
 
 \section{Background results}\label{background}

\subsection{Adjunction} 

\begin{lem}\label{adjlem}
Let $W$ be a quasi-projective scheme over $\IC$ and let $V$ be the total space of a line bundle $\CQ$ on $W$. Let $\sigma :W\to V$ denote the zero section. Let $C^{\bullet}$ be an  
 object of $\mathscr{D}^b(W)$.  There is an isomorphism  
\be\label{eq:splitadj}
L\sigma^*\dR_{\sigma*}C^{\bullet} \cong C^{\bullet} \oplus C^{\bullet}\otimes \CQ^{-1}[1] 
\ee
 in $\mathscr{D}^b(W)$. Moreover, given any objects $C^{\bullet}_1, C^{\bullet}_2$ in $\mathscr{D}^b(W)$, there is a splitting 
 \be\label{eq:rhomsplitA} 
 \bal 
 \dR\mathscr{H}om_{\mathscr{D}^{b}(V)}(\sigma_*C^{\bullet}_1, \sigma_*C^{\bullet}_2) \cong\ & \dR_{\sigma*}\dR\mathscr{H}om_{\mathscr{D}^{b}(W)}(C^{\bullet}_1,C^{\bullet}_2) \oplus \\
 & \dR_{\sigma*}\dR\mathscr{H}om_{\mathscr{D}^{b}(W)}(C^{\bullet}_1\otimes \CQ^{-1},C^{\bullet}_2)[-1]
 \eal
\ee
in $\mathscr{D}^b(W)$.
\end{lem}

{\it Proof}. 
Since $W$ is quasi-projective,  $C^{\bullet}$ has a quasi-coherent, possibly infinite, locally free resolution $F^{\bullet}_{C^{\bullet}}$. 
Let $p:V \to W$ be the canonical projection, which is affine and flat. Let $\zeta\in H^0(V, p^*\CQ)$ denote the tautological section. Then the locally free complex 
\[
 {\rm Cone} ({\bf 1}_{F^{\bullet}_{C^{\bullet}}}\otimes \zeta: p^*(F^{\bullet}_{C^{\bullet}} \otimes 
p^*\CL^{-1})\to p^*F^{\bullet}_{C^{\bullet}})
\]
is quasi-isomorphic to $\dR_{\sigma*}C$. Since the zero locus of $\zeta$ coincides with the image of the zero section $\sigma$, 
\[
L\sigma^*\dR_{\sigma*}C \cong \sigma^*F^{\bullet}_{C^{\bullet}} \oplus F^{\bullet}_{C^{\bullet}} \otimes 
\CQ^{-1} [1]. 
\]
This proves equation \eqref{eq:splitadj}. 

Next, \cite[Prop. II.5.10]{H66} shows that for any objects $C^{\bullet}_1,C^{\bullet}_2$ of  $\mathscr{D}^b(W)$ there is a natural isomorphism 
\[ 
\dR_{\sigma*}\dR\mathscr{H}om_{\mathscr{D}^{b}(W)}(L\sigma^*\dR_{\sigma*}C^{\bullet}_1, \dR_{\sigma*}C^{\bullet}_2) \cong \dR\mathscr{H}om_{\mathscr{D}^{b}(V)}(\dR_{\sigma*}C^{\bullet}_1, \dR_{\sigma*}C^{\bullet}_2).
\]
Then equation \eqref{eq:rhomsplitA} follows from \eqref{eq:splitadj}.

\hfill $\Box$

\subsection{Products and diagonals} 

This section consists of some background results on product schemes needed in the proof of Theorem 
\ref{Apropone}. 

\begin{lem}\label{flatlemma} 
Suppose $W_1, W_2$ are quasi-projective over $\IC$ and let
 $\mathcal{G}_1$ be a coherent sheaf on $W_1$.  Then 
$\pi_{W_1}^*\mathcal{G}_1$ is flat over $W_2$ and the derived pull-back 
$L\pi_{W_1}^*\mathcal{G}_1$ is isomorphic to the ordinary pull-back 
$\pi_{W_1}^*\mathcal{G}_1$. Moreover,  for any coherent sheaves $G_i$ on 
$W_i$ with $1\leq i \leq 2$, 
one has the transversality result 
\be\label{eq:transvA} 
{\mathcal Tor}^k_{W_1\times W_2}( \pi_{W_1}^*\mathcal{G}_1, \pi_{W_2}^*\mathcal{G}_2) 
=0
\ee
for $k\geq 1$. In particular, given any the derived tensor product $L\pi_{W_1}^*\mathcal{G}_1 \otimes^L L\pi_{W_2}^*\mathcal{G}_2$ is isomorphic to the ordinary tensor product $\pi_{W_1}^*\mathcal{G}_1 \otimes \pi_{W_2}^*\mathcal{G}_2$.
 \end{lem} 

{\it Proof.} This follows from the cartesian diagram 
\[
\xymatrix{ 
W_1 \times W_2 \ar[r]^-{\pi_{W_1}}\ar[d]_-{\pi_{W_2}} & W_1 \ar[d]^-{f_1} \\ 
W_2 \ar[r]^-{f_2} & {\rm Spec}(\IC). \\}
\]
Since $f_2$ is flat, there is an isomorphism $\pi_{{W_2}*}\pi_{W_1}^*\mathcal{G}_1\cong f_2^*f_{1*}\mathcal{G}_1$, while $f_{1*}\mathcal{G}_1$ is obviously locally free on ${\rm Spec}(\IC)$. Moreover, $\pi_{W_1}$ is flat since $f_2$ is flat, which implies the second statement.

Since both schemes are quasi-projective, the sheaves $\mathcal{G}_1, \mathcal{G}_2$  admit locally free resolutions. Then the transversality result follows immediately from the first two statements.  

\hfill $\Box$ 

\begin{lem}\label{flatcorA} 
Let $W_1,W_2$ be quasi-projective schemes over $\IC$ 
and let $W=W_1\times W_2$. Let $C^{\bullet}$ be a complex of coherent sheaves on $W_2$. Then there is a  canonical isomorphism
\be\label{eq:diagisomC}
\Delta_{W*}\pi_{W2}^*C^{\bullet}\cong \CO_{\Delta_{W_1}} \boxtimes \Delta_{W_2*}C^{\bullet}.
\ee
\end{lem} 

{\it Proof}. 
Consider the commutative diagram 
\[
\xymatrix{
W \ar[rr]^-{\Delta_{W}}\ar[d]^-{\pi_{W_2}}  & & W\times W \ar[d]^-{{\pi}_{W_2\times W_2}} \\
W_2 \ar[rr]^-{\Delta_{W_2}} & & W_2\times W_2. \\}
\]
and note the canonical isomorphism 
\[ 
C^{\bullet} \cong \Delta_{W_2}^*\Delta_{W_2*} C^{\bullet}.
\]
This implies 
\[
\pi_{W_2}^*C^{\bullet} \cong \pi_{W_2}^*\Delta_{W_2}^* \Delta_{W_2*}C^{\bullet} \cong \Delta_{W}^* \pi_{W_2\times W_2}^* \Delta_{W_2*}C^{\bullet}. 
\]
Hence 
\[ 
\bal 
\Delta_{W*}\pi_{W_2}^*C^{\bullet} & \cong \Delta_{W*}\Delta_{W}^* \pi_{W_2\times W_2}^* \Delta_{W_2*}C^{\bullet}  \cong \pi_{W_2\times W_2}^* \Delta_{W_2*}C^{\bullet}/ \mathcal{I}_{\Delta_{W} }
\pi_{W_2\times W_2}^* \Delta_{W_2*}C^{\bullet}\\
& \cong \CO_{\Delta_{W}}\otimes \pi_{W_2\times W_2}^* \Delta_{W_2*}C^{\bullet}\eal
\]
However, Lemma \ref{flatlemma} implies that the scheme theoretic intersection of $\pi_{W_1\times W_1}^{-1}(\Delta_{W_1})$ and $\pi_{W_2\times W_2}^{-1} (\Delta_{W_2})$ is transverse and coincides with $\Delta_{W}$. 
Therefore 
\[
\bal
\CO_{\Delta_{W}}\otimes \pi_{W_2\times W_2}^* \Delta_{W_2*}C^{\bullet}  
& \cong 
(\CO_{\Delta_{W_1}} \boxtimes \CO_{\Delta_{W_2}}) \otimes \pi_{W_2\times W_2}^* \Delta_{W_2*}C^{\bullet}\\
& \cong 
\CO_{\Delta_{W_1}} \boxtimes \pi_{W_2\times W_2}^* \Delta_{W_2*}C^{\bullet}.
\\
\eal
\]

\hfill $\Box$

\begin{lem}\label{flatcorB} 
Let $W_1, W_2$ be quasi-projective schemes over $\IC$ and 
let $W=W_1\times W_2$.  
Then there are 
 canonical exact sequences of $\CO_{W\times W}$-modules 
\be\label{eq:idealseqA} 
0 \to \mathcal{I}_{\Delta_{W_1}}\boxtimes \CO_{W_2 \times W_2} \xlongto{j_W}
\mathcal{I}_{\Delta_{W}} \xlongto{\gamma_W} \CO_{\Delta_{W_1}} \boxtimes \mathcal{I}_{\Delta_{W_2}} \to 0
\ee
and 
\be\label{eq:idealseqB} 
0 \to \mathcal{I}_{\Delta_{W_1}}\boxtimes \mathcal{I}_{\Delta_{W_2}}
\xlongto{\delta_W} \mathcal{I}_{\Delta_{W_1}} \boxplus  \mathcal{I}_{\Delta_{W_2}} \xlongto{\rho_W} 
\mathcal{I}_{\Delta_{W}} \to 0
\ee
where 
the $\delta_W$ is the antidiagonal map defined by multiplication, 
and $\gamma_W$, $\rho_W$ are the natural projections. Moreover, 
there is a commutative diagram 
\be\label{eq:idealdiagA} 
\xymatrix{ 
\mathcal{I}_{\Delta_{W_1}} \boxplus \mathcal{I}_{\Delta_{W_2}} \ar[r]^-{\rho_W}
\ar[d]^-{(0,\ f_{W_1}\boxtimes {\bf 1})}  
& \mathcal{I}_{\Delta_W} \ar[d]^-{\gamma_W} \\
\CO_{\Delta_{W_1}} \boxtimes \mathcal{I}_{\Delta_{W_2}} \ar[r]^-{\bf 1} 
& \CO_{\Delta_{W_1}} \boxtimes \mathcal{I}_{\Delta_{W_2}}, \\}
\ee
where $f_{W_1}: \CO_{W_1\times W_1} \to \CO_{\Delta_{W_1}}$ is the canonical projection.
\end{lem} 

{\it Proof.}  
Note that $\mathcal{I}_{\Delta_{W}} = \pi_{W_1\times W_1}^* \mathcal{I}_{\Delta_{W_1}} + \pi_{W_2\times W_2}^* \mathcal{I}_{\Delta_{W_2}}$, the sum of ideal sheaves on $W \times W$. In particular there is a canonical exact sequence
\[
0 \to \pi_{{W_2}\times {W_2}}^* \mathcal{I}_{\Delta_{W_2}} \to \mathcal{I}_{\Delta_{W}} \to 
\pi_{{W_1}\times {W_1}}^* \mathcal{I}_{\Delta_{W_1}}/ (\pi_{{W_1}\times {W_1}}^* \mathcal{I}_{\Delta_{W_1}}\cap \pi_{{W_2}\times {W_2}}^* \mathcal{I}_{\Delta_{W_2}}) \to 0.
\]
Moreover, Lemma \ref{flatlemma} shows that $\pi_{{W_1}\times {W_1}}^*\mathcal{I}_{\Delta_{W_1}}$ is flat over 
${W_2}\times {W_2}$, and the pull-back of the canonical exact sequence 
\[ 
0 \to \mathcal{I}_{\Delta_{W_2}} \to \CO_{{W_2}\times {W_2}} \to \CO_{\Delta_{W_2}} \to 0 
\]
to $W\times W$ remains exact. This yields an isomorphism 
\[ 
\pi_{{W_1}\times {W_1}}^* \mathcal{I}_{\Delta_{W_1}}/ (\pi_{{W_1}\times {W_1}}^* \mathcal{I}_{\Delta_{W_1}}\cap \pi_{{W_2}\times {W_2}}^* \mathcal{I}_{\Delta_{W_2}})\cong \mathcal{I}_{\Delta_{W_1}} \boxtimes \CO_{\Delta_{W_2}}. 
\]

For the second  sequence, the relation 
$\mathcal{I}_{\Delta_{W}} = \pi_{W_1\times W_1}^* \mathcal{I}_{\Delta_{W_1}} + \pi_{W_2\times W_2}^* \mathcal{I}_{\Delta_{W_2}}$ yields the canonical exact sequence 
\[
0 \to \pi_{W_1\times W_1}^*\mathcal{I}_{\Delta_{W_1}} \cap \pi_{W_2\times W_2}^* \mathcal{I}_{\Delta_{W_2}} \to \mathcal{I}_{\Delta_{W_1}} \boxplus  \mathcal{I}_{\Delta_{W_2}} \to 
\mathcal{I}_{\Delta_{W}} \to 0. 
\]
Moreover, Lemma \ref{flatlemma} implies that the multiplication map 
\[
\mathcal{I}_{\Delta_{W_1}}\boxtimes \mathcal{I}_{\Delta_{W_2}} \to 
\pi_{W_1\times W_1}^*\mathcal{I}_{\Delta_{W_1}} \cap \pi_{W_2\times W_2}^* \mathcal{I}_{\Delta_{W_2}}
\]
is an isomorphism. 

Diagram \eqref{eq:idealdiagA} follows immediately by construction. 

\hfill $\Box$

\begin{cor}\label{splitcor} 
Under the same conditions as in Lemma \ref{flatcorB} 
the
map $\rho_W$ in \eqref{eq:idealseqB} induces 
an isomorphism  
\be\label{eq:splitseq} 
{{\bar \rho_W}}: \mathcal{I}_{\Delta_{W_1}}/\mathcal{I}_{\Delta_{W_1}}^2 \boxtimes 
\CO_{\Delta_{W_2}} \oplus \CO_{\Delta_{W_1}} \boxtimes 
\mathcal{I}_{\Delta_{W_2}}/\mathcal{I}_{\Delta_{W_2}}^2 {\buildrel \sim \over \longto} 
\mathcal{I}_{\Delta_{W}}/\mathcal{I}_{\Delta_{W}}^2.
\ee
Moreover, there is a commutative diagram 
\be\label{eq:idealdiagC} 
\xymatrix{ 
\mathcal{I}_{\Delta_{W_1}}/\mathcal{I}_{\Delta_{W_1}}^2 \boxtimes 
\CO_{\Delta_{W_2}} \oplus \CO_{\Delta_{W_1}} \boxtimes 
\mathcal{I}_{\Delta_{W_2}}/\mathcal{I}_{\Delta_{W_2}}^2 \ar[r]^-{{\bar \rho}_W} 
\ar[d]^-{(0, {\bf 1})} &
\mathcal{I}_{\Delta_{W}}/\mathcal{I}_{\Delta_{W}}^2\ar[d]^-{{\bar\gamma}_W} \\
\CO_{\Delta_{W_1}} \boxtimes 
\mathcal{I}_{\Delta_{W_2}}/\mathcal{I}_{\Delta_{W_2}}^2 \ar[r]^-{\bf 1} & 
\CO_{\Delta_{W_1}} \boxtimes 
\mathcal{I}_{\Delta_{W_2}}/\mathcal{I}_{\Delta_{W_2}}^2 \\}
\ee
where ${\bar \gamma}_W: \mathcal{I}_{\Delta_{W}}/\mathcal{I}_{\Delta_{W}}^2 \to 
\CO_{\Delta_{W_1}} \boxtimes 
\mathcal{I}_{\Delta_{W_2}}/\mathcal{I}_{\Delta_{W_2}}^2$ is the natural projection determined by $\gamma_W: \mathcal{I}_{\Delta_W} \to \CO_{\Delta_{W_1}} \boxtimes 
\mathcal{I}_{\Delta_{W_2}}$. 
\end{cor}

{\it Proof}. 
Note that the image of the map $\delta_W$ in 
\eqref{eq:idealseqB} is contained in the subsheaf 
\[
(\mathcal{I}_{\Delta_{W_1}} \boxtimes \CO_{W_2\times W_2})\cap 
\mathcal{I}_{\Delta_W}^2  \oplus 
(\CO_{W_1\times W_1} \boxtimes \mathcal{I}_{\Delta_{W_2}})\cap 
\mathcal{I}_{\Delta_W}^2 \subset \mathcal{I}_{\Delta_{W_1}} \boxtimes \mathcal{I}_{\Delta_{W_2}}. 
\]
Then the exact sequence \eqref{eq:idealseqB} yields the following commutative diagram 
\be\label{eq:idealdiagB}
\xymatrix{ 
0 \ar[r] & \mathcal{I}_{\Delta_{W_1}} \boxtimes \mathcal{I}_{\Delta_{W_2}} 
\ar[r]^-{\delta_W'} \ar[d]^-{\bf 1}  & {\begin{array}{c} (\mathcal{I}_{\Delta_{W_1}} \boxtimes \CO_{W_2\times W_2})\cap 
\mathcal{I}_{\Delta_W}^2 \\ \oplus \\
(\CO_{W_1\times W_1} \boxtimes \mathcal{I}_{\Delta_{W_2}})\cap 
\mathcal{I}_{\Delta_W}^2 \end{array}} \ar[r]^-{\rho_W'} \ar[d] 
& \mathcal{I}_{\Delta_W}^2 \ar[r] \ar[d] & 0 \\
0 \ar[r] & \mathcal{I}_{\Delta_{W_1}} \boxtimes \mathcal{I}_{\Delta_{W_2}} 
\ar[r]^{\delta_W}  & \mathcal{I}_{\Delta_{W_1}}\boxplus \mathcal{I}_{\Delta_{W_2}} \ar[r]^-{\rho_W} & \mathcal{I}_{\Delta_W} \ar[r]  & 0\\}
\ee
with exact rows, where $\delta_W'$ is naturally determined by $\delta_W$ and $\rho'_W$ is the restriction of $\rho_W$.
Moreover, note that 
\[
\mathcal{I}_{\Delta_W}^2 = \mathcal{I}_{\Delta_{W_1}}^2 \boxtimes \CO_{W_2\times W_2} 
+  \mathcal{I}_{\Delta_{W_1}} \boxtimes \mathcal{I}_{\Delta_{W_2}} + \CO_{W_1\times W_1}\boxtimes \mathcal{I}_{\Delta_{W_2}}^2.
\]
Therefore there are canonical isomorphisms 
\[
(\mathcal{I}_{\Delta_{W_1}} \boxtimes \CO_{W_2\times W_2})\cap 
\mathcal{I}_{\Delta_W}^2 \cong \mathcal{I}_{\Delta_{W_1}}^2 \boxtimes \CO_{W_2\times W_2} 
+  \mathcal{I}_{\Delta_{W_1}} \boxtimes \mathcal{I}_{\Delta_{W_2}}
\]
and 
\[
(\CO_{W_1\times W_1} \boxtimes \mathcal{I}_{\Delta_{W_2}})\cap 
\mathcal{I}_{\Delta_W}^2 \cong \mathcal{I}_{\Delta_{W_1}} \boxtimes \mathcal{I}_{\Delta_{W_2}}
+ \CO_{W_1\times W_1} \boxtimes \mathcal{I}_{\Delta_{W_2}}^2 
\]
These yield further isomorphisms 
\[ 
\mathcal{I}_{\Delta_{W_1}}/(\mathcal{I}_{\Delta_{W_1}} \boxtimes \CO_{W_2\times W_2})\cap 
\mathcal{I}_{\Delta_W}^2 \cong  \mathcal{I}_{\Delta_{W_1}}/\mathcal{I}_{\Delta_{W_1}}^2
\boxtimes \CO_{W_2\times W_2}
\]
and 
\[
\mathcal{I}_{\Delta_{W_2}}/(\mathcal{I}_{\Delta_{W_1}} \boxtimes \CO_{W_2\times W_2})\cap 
\mathcal{I}_{\Delta_W}^2 \cong  \CO_{W_1\times W_1}\boxtimes 
 \mathcal{I}_{\Delta_{W_2}}/\mathcal{I}_{\Delta_{W_2}}^2
 \]
Then isomorphism \eqref{eq:splitseq}  follows from \eqref{eq:idealdiagB} by the snake lemma. Diagram \eqref{eq:idealdiagC} follows from \eqref{eq:idealdiagA}. 

\hfill $\Box$

\bibliography{CY4_ref.bib}
\bibliographystyle{abbrv}

\end{document}